\documentclass[twoside,11pt,reqno]{amsart}
\usepackage{amsmath,amssymb,amscd,mathrsfs,epic,wasysym,latexsym,tikz,mathrsfs,cite,hyperref,tensor,enumerate,graphicx}
\usepackage{stmaryrd}
\usepackage{pb-diagram}
\usepackage[matrix,arrow]{xy}

\makeatletter

\numberwithin{equation}{section}

\newtheorem{Proposition}[equation]{Proposition}
\newtheorem{Lemma}[equation]{Lemma}
\newtheorem{Theorem}[equation]{Theorem}
\newtheorem{Corollary}[equation]{Corollary}
\theoremstyle{definition}  
\newtheorem{Definition}[equation]{Definition}
\newtheorem{Remark}[equation]{Remark}

\let\<\langle
\let\>\rangle

\newcommand\Comment[2][\relax]{\space\par\medskip\noindent%
	\fbox{\begin{minipage}{\textwidth}\textbf{Comment\ifx\relax#1\else---#1\fi}\newline%
			#2\end{minipage}}\medskip
}

\def\balpha{\text{\boldmath$\alpha$}}

\def\bl{\text{\boldmath$l$}}
\def\bp{\text{\boldmath$p$}}
\def\bq{\text{\boldmath$q$}}
\def\bs{\text{\boldmath$s$}}

\def\bt{\text{\boldmath$t$}}
\def\bc{\text{\boldmath$c$}}
\def\br{\text{\boldmath$r$}}
\def\b1{\text{\boldmath$1$}}

\def\ba{\text{\boldmath$a$}}
\def\bb{\text{\boldmath$b$}}
\def\bw{\text{\boldmath$w$}}
\def\bu{\text{\boldmath$u$}}
\def\bv{\text{\boldmath$v$}}
\def\bx{\text{\boldmath$x$}}
\def\by{\text{\boldmath$y$}}

\def\balpha{\text{\boldmath$\alpha$}}

\def\biota{\text{\boldmath$\iota$}}

\def\coproduct{\reflectbox{\rotatebox[origin=c]{180}{${\tt\Delta}$}}}

\def\bsi{\text{\boldmath$\sigma$}}
\def\btau{\text{\boldmath$\tau$}}

\def\a{\mathfrak{a}}
\def\c{\mathfrak{c}}

\def\pmod#1{\text{ }(\text{\rm mod } #1)\,}

\newcommand{\Hom}{\operatorname{Hom}}

\newcommand{\id}{\operatorname{id}}

\newcommand{\head}{\operatorname{head}}

\newcommand{\Stab}{\operatorname{Stab}}

\newcommand{\Z}{\mathbb{Z}}

\newcommand{\F}{\mathbb{F}}

\newcommand{\0}{{\bar 0}}
\renewcommand{\1}{{\bar 1}}
\def\eps{{\varepsilon}}
\def\phi{{\varphi}}

\newcommand{\Triple}{{\mathcal T}}

\newcommand{\ga}{\gamma}
\newcommand{\Ga}{\Gamma}
\newcommand{\la}{\lambda}
\newcommand{\La}{\Lambda}

\def\Si{\mathfrak{S}}
\newcommand{\si}{\sigma}
\newcommand{\om}{\omega}
\newcommand{\Om}{\Omega}

\newcommand{\de}{\delta}
\newcommand{\De}{\Delta}
\newcommand{\ka}{\kappa}

\def\id{\mathop{\mathrm {id}}\nolimits}

\newcommand{\tr}{{\mathsf {con}}}

\newcommand{\Sym}{{\mathrm {Sym}}}

\newcommand{\D}{{\mathscr D}}

\renewcommand{\mod}{\bmod \,}

\newcommand{\EZig}{{\mathsf Z}}
\renewcommand{\Alph}{{\mathscr A}}

\def\col{{\operatorname{col}}}
\def\letter{{\operatorname{let}}}

\newcommand{\Std}{\operatorname{Std}}

\def\Brackets#1{[ #1 ]}

\def\Par{{\mathscr P}}

\def\b{\mathfrak{b}}
\def\k{\Bbbk}

\def\T{\text{\boldmath$T$}}

\def\Stab{\text{\boldmath$S$}}

\def\spa{\operatorname{span}}

\def\op{{\mathrm{op}}}

\def\onto{{\twoheadrightarrow}}

\def\mod#1{#1\!\operatorname{-mod}}

\def\iso{\stackrel{\sim}{\longrightarrow}}

\def\X{{\mathcal X}}
\def\Y{{\mathcal Y}}

\def\ch{\operatorname{ch}}
\def\lan{\langle}
\def\ran{\rangle}

\def\Seq{\operatorname{Tri}}

\def\bla{\text{\boldmath$\lambda$}}
\def\bmu{\text{\boldmath$\mu$}}
\def\bnu{\text{\boldmath$\nu$}}

\def\bga{\text{\boldmath$\gamma$}}

{\catcode`\|=\active
	\gdef\set#1{\mathinner{\lbrace\,{\mathcode`\|"8000%
				\let|\midvert #1}\,\rbrace}}
}
\def\midvert{\egroup\mid\bgroup}

\colorlet{darkgreen}{green!50!black}
\tikzset{dots/.style={very thick,loosely dotted},
	greendot/.style={fill,circle,color=darkgreen,inner sep=1.5pt,outer sep=0},
	blackdot/.style={fill,circle,color=black,inner sep=1.5pt,outer sep=0},
	graydot/.style={fill,circle,color=gray,inner sep=1.1pt,outer sep=0}
}
\def\greendot(#1,#2){\node[greendot] at(#1,#2){}}
\def\blackdot(#1,#2){\node[blackdot] at(#1,#2){}}
\def\graydot(#1,#2){\node[graydot] at(#1,#2){}}

\newenvironment{braid}{
	\begin{tikzpicture}[baseline=6mm,black,line width=1pt, scale=0.32,
		draw/.append style={rounded corners},
		every node/.append style={font=\fontsize{5}{5}\selectfont}]%
	}{\end{tikzpicture}
}

\def\Grid(#1,#2){
	\draw[very thin,gray,step=2mm] (0,0)grid(#1,#2);
	\draw[very thin,darkgreen,step=10mm] (0,0)grid(#1,#2);
}

\newcommand\Tableau[2][\relax]{
	\begin{tikzpicture}[scale=0.5,draw/.append style={thick,black}]
		\ifx\relax#1\relax%
		\else 
		\foreach\box in {#1} { \filldraw[blue!30]\box+(-.5,-.5)rectangle++(.5,.5); }
		\fi
		\newcount\row\newcount\col
		\row=0
		\foreach \Row in {#2} {
			\col=1
			\foreach\k in \Row {
				\draw(\the\col,\the\row)+(-.5,-.5)rectangle++(.5,.5);
				\draw(\the\col,\the\row)node{\k};
				\global\advance\col by 1
			}
			\global\advance\row by -1
		}
	\end{tikzpicture}
}

\newcommand\YoungDiagram[2][\relax]{
	\begin{tikzpicture}[scale=0.5,draw/.append style={thick,black}]
		\ifx\relax#1\relax%
		\else 
		\foreach\box in {#1} {
			\filldraw[blue!30]\box rectangle ++(1,1);
		}
		\fi
		\newcount\row
		\row=0
		\foreach \col in {#2} {
			\draw(1,\the\row)grid ++(\col,1);
			\global\advance\row by -1
		}
	\end{tikzpicture}
}

\begin{document}
	
	\title[Good filtrations]{{\bf Good filtrations for generalized Schur algebras}}
	
	\author{\sc Alexander Kleshchev}
	\address{Department of Mathematics\\ University of Oregon\\
		Eugene\\ OR 97403, USA}
	\email{klesh@uoregon.edu}
	
	\author{\sc Ilan Weinschelbaum}
	\address{Department of Mathematics\\ University of Oregon\\
		Eugene\\ OR 97403, USA}
	\email{ilanw@uoregon.edu}
	
\dedicatory{To the memory of James Humphreys}
	
	
\thanks{The first author was supported by the NSF grant DMS-2101791.}

	\begin{abstract}
		Given a quasi-hereditary superalgebra $A$, the first author and R. Muth have defined generalized Schur bi-superalgebras $T^A(n)$ and proved that these algebras are again  quasi-hereditary. In particular, $T^A(n)$ comes with a family of standard modules. Developing the work of Donkin and Mathieu on good filtrations, we prove that tensor product of standard modules over $T^A(n)$ has a standard filtration. 
	\end{abstract}
	
	\maketitle
	
	\section{Introduction}
	Let $S(n,d)$ be the classical Schur algebra as in \cite{Green}. A fundamental fact going back to \cite{GreenComb} is that the algebra $S(n,d)$ is (based) quasi-hereditary, and so by \cite{CPS}, the module category  $\mod{S(n,d)}$ is a highest weight category, cf. \cite[(2.5.3)]{Parshall}, with standard modules $\{\De(\la)\mid\la\in\La_+(n,d)\}$, where $\La_+(n,d)$ is the set of partitions of $d$ with at most $n$ parts. Moreover, there is a coproduct $\coproduct$ on $S(n)=\bigoplus_{d= 0}^{\infty}S(n,d)$, which  makes it into a bialgebra. Upon restriction, 
	$$\coproduct: S(n,d)\to\bigoplus_{d_1+d_2=d}S(n,d_1)\otimes S(n,d_2),$$ 
	so for $\la\in\La_+(n,d_1)$ and $\mu\in\La_+(n,d_2)$, the tensor product $\De(\la)\otimes \De(\mu)$ becomes an $S(n,d_1+d_2)$-module. It follows from (the easy type $A$ case of) the work of Donkin \cite{DonkinFilt} and Mathieu \cite{Math} that $\De(\la)\otimes \De(\mu)$ has a standard filtration, i.e. a filtration whose subquotients are standard modules $\De(\nu)$  with $\nu\in\La_+(n,d_1+d_2)$.

	Let $R$ be a characteristic $0$ 
domain and $\F$ be a field which is an $R$-module. For the purposes of this introduction, we can take $R=\Z$ and $\F$ any field. 
	Let $A$ be a (based) quasi-hereditary superalgebra over $R$, i.e.  $A$ is endowed with a heredity data $I,X,Y$, see \S~\ref{SSBQHA} for details. We make a technical assumption that the data $I,X,Y$ are conforming---this assumption is vacuous in the purely even situation and holds in all the important examples known to us, for example when $A$ is an extended zigzag superalgebra. 
	
	Developing a construction of Turner \cite{T,T2,T3}, the first  author and Muth defined generalized Schur (super)algebras $$T^A(n)=\bigoplus_{d= 0}^{\infty}T^A(n,d),$$ 
see \cite{KMgreen2},  cf. also \cite{EK1}. Importantly, we must first define $R$-forms  
	$T^A(n,d)_R$ of the algebras and then extend scalars to get $T^A(n,d)=\F\otimes_RT^A(n,d)_R$, 
	so it is crucial that $A$ is defined over $R$. When $A$ is the trivial algebra $R$, the generalized Schur algebra returns the classical Schur algebra $S(n,d)$. 
	
As predicted by Turner, certain generalized Schur algebras play an important role in modular representation theory, see \cite{T,EK2}. For example, it is conjectured in \cite[Conjecture 7.61]{KMgreen3} that weight $d$\, RoCK blocks of the classical Schur algebras and $q$-Schur algebras are Morita equivalent to the generalized Schur algebras of the form $T^{\EZig}(n,d)$, where $\EZig$ is an extended zigzag algebra.

	It is proved in \cite{KMgreen3} that under the assumption $d\leq n$ the algebras $T^A(n,d)$ are again based quasi-hereditary, with heredity data $\La^I_+(n,d),\X,\Y$, where $\La^I_+(n,d)$ is the set of $I$-multipartitions of $d$, see \S\ref{SSQHT} for details. In particular, the category of finite dimensional $T^A(n,d)$-modules is a highest weight category with  standard and costandard modules 
	$$\{\De(\bla)\mid\bla\in\La_+^I(n,d)\}
	\quad\text{and}\quad
	\{\nabla(\bla)\mid\bla\in\La_+^I(n,d)\},
	$$ 
	respectively. 
If $d>n$, the `standard' modules $\De(\bla)$ and the `costandard'  modules $\nabla(\bla)$ are still defined and play an important role. For example, if $A$ has a standard anti-involution then $T^A(n,d)$ is cellular with the cell modules $\De(\bla)$, see \cite[Lemma 6.25]{KMgreen3}. 
	
	There is a natural coproduct $\coproduct$ on $T^A(n)$, with  
	$$\coproduct: T^A(n,d)\to\bigoplus_{d_1+d_2=d}T^A(n,d_1)\otimes T^A(n,d_2).$$ 
So for $\bla\in\La^I_+(n,d_1)$ and $\bmu\in\La^I_+(n,d_2)$, the tensor product $\De(\bla)\otimes \De(\bmu)$ becomes a  $T^A(n,d_1+d_2)$-module. The main goal of this note is to prove the following
	
	\vspace{2mm}
	\noindent
	{\bf Main Theorem.}
	{\em 
		The tensor product of two standard (resp. costandard) modules over $T^A(n)$ has a standard (resp. costandard) filtration. 
	}
	
\vspace{2mm}

	We hope that this theorem might be useful to study Ringel and Koszul duality for generalized Schur algebras. In particular, we are interested in 
generalizing the results of Donkin \cite{DonkinTilt,DonkinQS} on Ringel duality for classical (and $q$-) Schur algebras. 
In the follow up paper \cite{KW2}, we will use the Main Theorem to prove that the extended zigzag Schur algebra $T^\EZig(n,d)$ is Ringel self-dual. 
	
	The paper is organized as follows. Section~\ref{SPrelim} is preliminary. In particular, in \S\ref{SSBQHA}, we review the theory of based quasi-hereditary algebras, and in \S\S\ref{SSComb},\ref{SSTab}, we introduce the necessary combinatorial notions needed to handle the formal characters of standard modules over generalized Schur algebras. In Section~\ref{SSA}, we begin to work with generalized Schur algebras. After reviewing the definition and main properties, in \S\ref{SSChF} we prove a character identity which is a combinatorial version of the standard filtration of the tensor product of two standard modules. In Section~\ref{SMain}, we prove the Main Theorem. Our argument, when specialized to the case $A=R$, yields a very easy proof for the classical Schur algebra case which seems to be new (or at least not well-known).

\section{Preliminaries}
	\label{SPrelim}
	\subsection{General notation}
	\label{SGeneralNot}
	For $m,n\in\Z$, we denote $[m,n]:=\{k\in\Z\mid m\leq k\leq n\}$. If 
	$n\in\Z_{>0}$, we also denote $[n]:=\{1,2,\dots,n\}$.
	
	Throughout the paper, $I$ denotes a finite partially ordered set. We always identify $I$ with the set $\{0,1,\dots,l\}$ for $l=|I|-1$, so that the standard total order on integers refines the partial order on $I$.

	For a set $S$, we often write elements of $S^d$ as words $\bs=s_1\cdots s_d$ with $s_1,\dots,s_d\in S$. 
	The symmetric group $\Si_d$ acts on the right on $S^d$ by place permutations:
	$$
	(s_1\cdots s_d)\si=s_{\si 1}\cdots s_{\si d}.
	$$

	An (arbitrary) ground field is denoted by $\F$. 
	Often we will also need to work over a characteristic $0$  domain $R$ such that $\F$ is a $R$-module, so that we can change scalars from $R$ to $\F$ (in all examples of interest to us, one can use $R=\Z$). We use $\k$ to denote $\F$ or $R$ and use it whenever the nature of the ground ring is not important. On the other hand, when it is important to emphasize whether we are working over $R$ or $\F$, we will use lower indices; for example for an $R$-algebra $A_R$ and an $A_R$-module $V_R$, after extending scalars we have that $V_\F:=\F\otimes_RV_R$ is a module over $A_\F:=\F\otimes_RA_R$.

\subsection{Superspaces and supermodules}
	\label{SSSpaces}
	Let $V=\bigoplus_{\eps\in\Z/2}V_\eps$ be a $\k$-supermodule. 
	If $v\in V_\eps\setminus\{0\}$ for $\eps\in\Z/2$, we say $v$ is {\em homogeneous}, write $|v|=\eps$, and refer to $\eps$ as the {\em parity} of $v$. 
	If $S\subseteq V$, we denote $S_\0:=S\cap V_\0$ and $S_\1:=S\cap V_\1$. If $S$ consists of homogeneous elements then $S=S_\0\sqcup S_\1$. 
Let $V$ and $W$ be superspaces. For $\de\in\Z/2$, a parity $\de$ (homogeneous) linear map $f:V\to W$ is a linear map satisfying $f(V_{\eps})\subseteq W_{\eps+\de}$ for all $\eps$.

	Let $d\in\Z_{>0}$. The group $\Si_d$ acts on $V^{\otimes d}$ on the right with   automorphisms, such that for all homogeneous $v_1,\dots,v_d\in V$ and $\si\in \Si_d$, we have 
	\begin{equation}\label{ESiAct}
		(v_1\otimes\dots\otimes v_{d})^\si=
		(-1)^{\lan\si;\bv\ran} v_{\si1}\otimes\dots\otimes v_{\si d},
	\end{equation}
	where, setting $\bv:=v_1\cdots v_d\in V^d$, we have put:
	\begin{equation}\label{EAngleSi}
		\lan\si;\bv\ran:=\sharp\{(k,l)\in[d]^2\mid k<l,\,  \si^{-1}k>\si^{-1}l,\ \text{and}\ v_k,v_l\in V_\1\}.
	\end{equation}
	
	For $0\leq c\leq d$, denote by ${}^{(c,d-c)}\D$ the set of the shortest coset representatives for $(\Si_c\times \Si_{d-c})\backslash\Si_{d}$. Given $w_1\in V^{\otimes c}$ and $w_2\in V^{\otimes (d-c)}$, we define the {\em star product}
	\begin{equation}\label{EStarNotationGen}
		w_1* w_2:=\sum_{\si\in{}^{(c,d-c)}\D}(w_1\otimes w_2)^\si\in V^{\otimes d}.
	\end{equation}

	Let $V$ and $W$ be superspaces, $d\in\Z_{\geq 0}$, and $\bv=v_1\cdots v_d\in V^d$ and $\bw=w_1\cdots w_d\in W^d$ be $d$-tuples of homogeneous elements. We denote 
	\begin{equation} 
		\label{EAngle2}\lan \bv, \bw\ran:=\sharp\{(k,l)\in[d]^2\mid k>l,\  v_k\in V_\1,w_l\in W_\1\}.
	\end{equation}
	
	Let now $A$ be  a unital $\k$-superalgebra. 
	As usual, the tensor product $A^{\otimes d}$ is a superalgebra with respect to 
	$$
	(a_1\otimes\dots\otimes a_d)(b_1\otimes\dots b_d)=(-1)^{\lan\ba,\bb\ran}a_1b_1\otimes\dots\otimes a_db_d,
	$$
	where we have put $\ba:=a_1\cdots a_d$, $\bb:= b_1\cdots b_d$ (here and below, in expressions like this, we assume that all elements are homogeneous).
	
	For any superspace $V$, we consider the subspace of invariants 
	\begin{equation}\label{EGa}
		\Ga^dV:=\big(V^{\otimes d}\big)^{\Si_d}=\{w\in V^{\otimes d}\mid w^\si=w\ \text{for all $\si\in\Si_d$}\}.
	\end{equation}
	If $A$ is a superalgebra, then $\Ga^dA$ is a subsuperalgebra of $A^{\otimes d}$. 
	
	Let $A$ be  a unital $\k$-superalgebra and $V,W$ be $A$-supermodules. A homogeneous {\em $A$-supermodule homomorphism} $f:V\to W$ is a homogeneous linear map $f:V\to W$ satisfying $f(av)=(-1)^{|f||a|}af(v)$ for all (homogeneous) $a,v$. Let 
$$\Hom_A(V,W)=\Hom_A(V,W)_\0\oplus \Hom_A(V,W)_\1$$ be the superspace of all $A$-supermodule homomorphisms from $V$ to $W$. 
We denote by $\mod{A}$ the category of all finitely generated (left) $A$-supermodules and all $A$-supermodule homomorphisms. We denote by `$\cong$' an isomorphism in this category and by `$\simeq$' an {\em even isomorphism} in this category.

We have the parity change functor $\Pi$ on $\mod{A}$: 
for $V\in\mod{A}$ we have $\Pi V\in\mod{A}$ with $(\Pi V)_\eps=V_{\eps+\1}$ for all $\eps\in \Z/2$ and the new action 
$a\cdot v=(-1)^{|a|}av$ for $a\in A, v\in V$. We have $V\cong \Pi V$ via the identity map.

	
	All subspaces, ideals, submodules, etc. are assumed to be homogeneous. 
	For example, given homogeneous elements $v_1,\dots,v_k$ of an $A$-supermodule $V$, we have the $A$-submodule  $A\lan v_1,\dots,v_k\ran\subseteq V$ generated by $v_1,\dots,v_k$.

\subsection{Based quasi-hereditary algebras}
	\label{SSBQHA}
	The main reference here is \cite{KMgreen1}. 
	Let $A$ be a $\k$-superalgebra.

	\begin{Definition} \label{DCC} {\rm \cite{KMgreen1}} 
		{\rm 
			Let $I$ be a finite partially ordered set and let $X=\bigsqcup_{i\in I}X(i)$ and $Y=\bigsqcup_{i\in I}Y(i)$ be finite sets of homogeneous elements of $A$ with distinguished {\em initial elements} $e_i\in X(i)\cap  Y(i)$ for each $i\in I$.
			For each $i\in I$, we set 
			$
			A^{>i}:=\spa\{xy\mid j>i,\, x\in X(j),\, y\in Y(j)\}$. 
			We say that $I,X,Y$ is  {\em heredity data} if the following axioms hold: 
			\begin{enumerate}
				\item[{\rm (a)}] $B:=\{x y \mid i\in I,\, x\in X(i),\, y\in Y(i)\}$ is a basis of $A$; 
				
				\item[{\rm (b)}] For all $i\in I$, $x\in X(i)$, $y\in Y(i)$ and $a\in A$, we have
				$$
				a x \equiv \sum_{x'\in X(i)}l^x_{x'}(a)x' \pmod{A^{>i}}
				\ \ \text{and}\ \ 
				ya \equiv \sum_{y'\in Y(i)}r^y_{y'}(a)y' \pmod{A^{>i}}
				$$
				for some $l^x_{x'}(a),r^y_{y'}(a)\in\k$;

				\item[{\rm (c)}] For all $i,j\in I$ and $x\in X(i),\ y\in Y(i)$ we have  
				\begin{align*}
					&xe_i= x,\ e_ix= \de_{x,e_i}x,\ e_i y= y,\ ye_i= \de_{y,e_i}y 
					\\
					&e_jx=x\ \text{or}\ 0,\ ye_j=y\ \text{or}\ 0. 
				\end{align*}
			\end{enumerate}
			
			If $A$ is endowed with heredity data $I,X,Y$, we call $A$ {\em based quasi-hereditary}, and refer to $B$ as a {\em heredity basis} of $A$. 
		}
	\end{Definition}

	From now on, $A$ is a  based quasi-hereditary superalgebra with heredity data $I,X,Y$. We refer to the idempotents $e_i$ as the {\em standard idempotents} of the heredity data. 
	We have $B=B_\0\sqcup B_\1$ and 
	\begin{equation}\label{EBABC}
		B_\0=B_\a\sqcup B_\c,
	\end{equation} 
	where 
	$$B_\a:=\{xy\mid i\in I, x\in X(i)_\0,y\in Y(i)_\0\},\quad B_\c:=\{xy\mid i\in I, x\in X(i)_\1,y\in Y(i)_\1\}.
	$$
	The heredity data $I,X,Y$ of $A$ is called {\em conforming} if $B_\a$ spans a unital subalgebra of $A$. 
	
	\begin{Lemma} \label{idemactionNew} {\rm \cite[Lemmas 2.7, 2.8]{KMgreen1}}
		Let $i,j\in I$ and \(x \in X(i)\), \(y \in Y(i)\). 
		\begin{enumerate}
			\item[{\rm (i)}] $e_ie_j=\de_{i,j}e_i$
			\item[{\rm (ii)}] If \(j \not \leq i\), then \(e_j x = y e_j = 0\).
		\end{enumerate}
	\end{Lemma}

	\begin{Corollary} 
		We have $X\cap Y=\{e_i\mid i\in I\}$.
	\end{Corollary}
	\begin{proof}
		Let $z \in X \cap Y$. As $z \in X$ we have $z\in X(i)$ so $ze_i = z$ for some $i\in I$. As $z \in Y$, we have $z\in Y(j)$ so $e_jz = z$ for some $j\in I$. By Lemma~\ref{idemactionNew}(ii), $j = i$, and the result follows from Definition~\ref{DCC}(c).
	\end{proof}
	
	
	\begin{Definition}
		Let $0\neq V\in\mod{A}$ and $i\in I$. We call $V$ a {\em highest weight module (of weight $i$)} if there exists a homogeneous $v\in V$ such that $e_iV$ is spanned by $v$, $Av=V$,  and $j> i$ implies $e_j V = 0$. In this case we refer to $v$ as a {\em highest weight vector} of $V$. 
	\end{Definition}
	
	\begin{Lemma} \label{LYCrit} 
		Let $i\in I$, $0\neq V\in\mod{A}$ and $v\in V$ be  a homogeneous vector. Suppose that $e_iv=v$, $Av=V$, and $yv=0$ for all $y\in Y\setminus\{e_i\}$. Then  $V$ is a highest weight module of weight $i$. 
	\end{Lemma}
	\begin{proof}
		Since $A$ is based quasi-hereditary, it follows from the assumption $yv=0$ for all $y\in Y\setminus\{e_i\}$ that $V$ is spanned by $\{xv\mid x\in X(i)\}$. The result now follows from Definition~\ref{DCC}(c) and Lemma~\ref{idemactionNew}(ii). 
	\end{proof}

	Fix $i\in I$. Note that $A^{>i}$ is an ideal in $A$ and denote 
	$\tilde A:=A/A^{>i}$, $\tilde a:=a+A^{>i}\in \tilde A$ for $a\in A$.  
	By inflation, $\tilde A$-(super)modules will be automatically considered as $A$-(super)modules. In particular, the {\em standard module} 
$$\De(i):=\tilde A \tilde e_i$$ 
is considered as an $A$-module. We have that $\De(i)$ is a free $\k$-module with basis $\{v_x:=\tilde x\mid x\in X(i)\}$ and the action $av_x=\sum_{x'\in X(i)}l_{x'}^x(a)v_{x'},
$ cf. \cite[\S2.3]{KMgreen1}. Denoting 
$$v_i:=v_{e_i}\in\De(i),$$ 
we have $e_iv_i=v_i$, and $e_j\De(i)\neq 0$ implies $j\leq i$  thanks to Lemma~\ref{idemactionNew}. Moreover, for all for all $x\in X(i)$ we have  $xv_i=v_x$, $e_i v_x=\de_{x,e_i}v_x$. In particular, $\De(i)$ is a highest weight module of weight $i$ (with even highest weight vector). If $V\in \mod{A}$ is isomorphic to $\De(i)$, then it is easy to see, using the fact that $e_iV$ is free of rank $1$ as a $\k$-module, that either $V\simeq \De(i)$ or $V\simeq \Pi\De(i)$. 

We also have the right standard $A$-module 
$$\De^\op(i):= \tilde e_i \tilde A,$$ 
and by symmetry every result we have about $\De(i)$ has its right analogue for $\De^\op(i)$, for example $\De^\op(i)$ is a free $\k$-module with basis $\{w_y:=\tilde y\mid y\in Y(i)\}$. 
	
	Let $V\in\mod{A}$. A {\em standard filtration} 
	of $V$ is an $A$-supermodule filtration 
	$0=W_0\subseteq W_1\subseteq \dots\subseteq W_l=V$ such that for every  $r=1,\dots,l$, we have $W_r/W_{r-1}\cong \De(i_r)$ for some $i_r\in I$. 
	We refer to $\De(i_1),\dots, \De(i_l)$ as the factors of the filtration, and to $\De(i_1)$ (resp. $\De(i_l)$) as the bottom (resp. top) factor. 
	
	Suppose now until the end of the subsection that $\k=\F$. 
	Then each $L(i):=\head \De(i)$ is irreducible, and $\{L(i)\mid i\in I\}$ is a complete set of non-isomorphic irreducible $A$-supermodules. We also have that $L^\op(i):=\head \De^\op(i)$ is an irreducible right module, and $\{L^\op(i)\mid i\in I\}$ is a complete set of non-isomorphic irreducible right $A$-supermodules.

	By \cite[Lemma 3.3]{KMgreen1}, $A$ is quasi-hereditary in the sense of Cline, Parshall and Scott, and $\mod{A}$ is a highest weight category with standard modules $\{\De(i)\mid i\in I\}$, see \cite[Theorem 3.6]{CPS}. In particular, the projective cover $P(i)$ of $L(i)$ has a standard filtration with the top factor $\De(i)$ and all other factors of the form $\De(j)$ or $\Pi\De(j)$ for $j>i$. Moreover, $\De(i)$ is the largest quotient of $P(i)$ such that $[\De(i):L(i)]=1$ and $[\De(i):L(j)]\neq 0$ implies $j\leq i$.

	\begin{Proposition} \label{PUn} 
		{\bf (Universality of standard modules)}
		Let $\k=\F$, $i\in I$, and $V$ be a highest weight module of weight $i$ with highest weight vector $v$. Then there is an homogeneous surjection $\De(i)\onto V$ of parity $|v|$; in particular $e_j V \neq 0$ implies $j\leq i$. 
	\end{Proposition}
	\begin{proof}
		Let $e_iV$ be spanned by $v\in V$. There is a homogeneous surjective $A$-supermodule homomorphism $\phi:Ae_i\onto V,\ ae_i\mapsto av$ of parity $|\phi|=|v|$. 
		As $e_iL(i)$ is $1$-dimensional and $e_iL(j)\neq 0$ implies $i\leq j$, we have that $Ae_i=P(i)\oplus P$, where $P$ is a direct sum of supermodules isomorphic to $P(j)$ with $j>i$. 
		
		Note that $\Hom_A(\De(j),V)=0$ for any $j>i$, so $\Hom_A(P(j),V)=0$ for all $j>i$, and we deduce $\Hom_A(P,V)=0$. So the map $\phi$ factors to give a surjection $P(i)\onto V$. Moreover, $P(i)$ has a standard filtration with top factor $\De(i)$ and other factors isomorphic to $\De(j)$ with $j>i$, so the map further factors through the surjection $\De(i)\onto V$.
	\end{proof}
	
	The following is a useful criterion for $V$ to have a standard filtration.

	\begin{Corollary}\label{FiltCriteria}
		Let $\k=\F$, $V\in\mod{A}$, $v_1\dots,v_t\in V$ be homogeneous elements, and set $V_s:=A\lan v_1,\dots,v_s\ran$ for $s=1,\dots, t$. Suppose that the following conditions hold:
		\begin{enumerate}
			\item[{\rm (1)}] $V_t=V$;
			\item[{\rm (2)}] for each $s=1,\dots,t$ there exists $i_s\in I$ such that $e_{i_s}v_s- v_s\in V_{s-1}$ and $yv_s\in V_{s-1}$ for all $y\in Y\setminus\{e_{i_s}\}$;
			\item[{\rm (3)}] $\dim V=\sum_{s=1}^t\dim\De(i_s)$.
		\end{enumerate}
		Then $V_s/V_{s-1}\simeq \Pi^{|v_s|}\De(i_s)$ for all $s=1,\dots,t$. 
		In particular, $V$ has a standard filtration. 
	\end{Corollary}
	\begin{proof}
		By Lemma~\ref{LYCrit} and Proposition~\ref{PUn}, each $V_s/V_{s-1}$ is a quotient of $\Pi^{|v_s|}\De(i_s)$. The result follows by dimensions. 
	\end{proof}
	
	The highest weight category $\mod{A}$ comes with costandard modules $\{\nabla(i)\mid i\in I\}$. Let $J(i)$ be the injective hull of $L(i)$ in $\mod{A}$ for $i\in I$. One can define $\nabla(i)$ as the largest submodule of $J(i)$ such that $[\nabla(i):L(i)]=1$ and $[\nabla(i):L(j)]>0$ implies $j\leq i$. Let $V\in\mod{A}$. 
	A {\em costandard filtration} 
	of $V$ is an $A$-supermodule filtration 
	$0=W_0\subseteq W_1\subseteq \dots\subseteq W_l=V$ such that for every  $r=1,\dots,l$, we have $W_r/W_{r-1}\cong \nabla(i_r)$ 
	for some $i_r\in I$. 
	
	Given a right $A$-supermodule $V$, there is a (left) $A$-supermodule structure on $V^*$ with $af(v)=(-1)^{|a||f|+|a||v|}f(va)$ for $a\in A,f\in V^*,v\in V$. 
	For example, note that $L^\op(i)^*$ is irreducible, $e_iL^\op(i)^*\neq 0$, and $e_jL^\op(i)^*\neq 0$ implies $j\leq i$;  therefore $L^\op(i)^*\simeq L(i)$. Denoting by $P^\op(i)$ the projective cover of $L^\op(i)$, we deduce that $P^\op(i)^*\simeq J(i)$. This in turn implies easily:
	\begin{equation}\label{ENabla}
		\nabla(i)\simeq\De^\op(i)^*.
	\end{equation}
	
	If  $0=W_0\subseteq W_1\subseteq \dots\subseteq W_l=V$ 
	is a standard filtration of $V$ and $i\in I$ then $\sharp\{1\leq r\leq l\mid W_r/W_{r-1}\cong \De(i)\}$ does not depend on the choice of the standard filtration and is denoted $(V:\De(i))$. In fact, by \cite[Proposition A2.2]{DonkinQS}, we have
	\begin{equation}\label{EDeMult}
		(V:\De(i))=\dim\Hom_A(V,\nabla(i)).
	\end{equation}

\subsection{Partitions and compositions}\label{SSComb}
	We denote by $\La_+$ the set of all partitions. For $\la\in\La_+$,   we have the conjugate partition $\la'$, see \cite[p.2]{Mac}. The Young diagram of $\la$ is 
$$[\la]:=\{(r,s)\in \Z_{>0}\times\Z_{>0} \mid s\leq \la_r\}.$$ 
We refer to $(r,s)\in[\la]$ as the {\em nodes} of $\la$. 
Define the partial order $\preceq$ on the nodes as follows: $(r,s)\preceq(r',s')$ if and only if $r\leq r'$ and $s\leq s'$. 

	For $\la,\mu,\nu\in\La_+$, we denote by $c^{\,\la}_{\mu,\nu}$ the corresponding Littlewood-Richardson coefficient, see \cite[\S\,I.9]{Mac}. 
	
	Let $n\in \Z_{>0}$. We denote $\La(n)=\Z_{\geq 0}^n$ and interpret it as the set of {\em compositions}\, $\la=(\la_1,\dots,\la_n)$ with $n$ non-negative parts. For $\la,\mu\in\La(n)$, we define 
$$\la+\mu:=(\la_1+\mu_1,\dots,\la_n+\mu_n).$$ 
For $1\leq r\leq n$, we denote 
\begin{equation}\label{EEps}
\eps_r:=(0,\dots,0,1,0,\dots,0)\in\La(n)
\end{equation}
with $1$ in position $r$. 
For $\la=(\la_1,\dots,\la_n)\in\La(n)$, set $|\la|:=\la_1+\dots+\la_n$. 

Denote 
$$\La_+(n):=\{\la=(\la_1,\dots,\la_n)\in\La(n)\mid\la_1\geq\dots\geq\la_n\}.$$ 
Sometimes we collect equal parts of $\la\in \La_+(n)$ to write it as $\la=(l_1^{a_1},\dots,l_k^{a_k})$ for $l_1>\dots>l_k\geq 0$ and $a_1,\dots,a_k>0$ with $a_1+\dots+a_k=n$. 
We interpret $\La_+(n)$ as a subset of $\La_+$ in the obvious way. 
For $d\in\Z_{\geq 0}$, let 
\begin{align*}
\La(n,d)=\{\la\in\La(n)\mid|\la|=d\}  
\quad \text{and}\quad
\La_+(n,d)=\{\la\in\La_+(n)\mid|\la|=d\}.
\end{align*}
	
	Let $S$ be a finite set. 
	We will consider the set of {\em $S$-multicompositions} and {\em $S$-multipartitions} 
	\begin{eqnarray*}\label{ELaX}
\La^S(n)&:=&\La(n)^S=\{\bla=(\la^{(s)})_{s\in S}\mid \la^{(s)}\in\La(n)\ \text{for all $s\in S$}\},
\\
\La_+^S(n)&:=&\La_+(n)^S=\{\bla=(\la^{(s)})_{s\in S}\mid \la^{(s)}\in\La_+(n)\ \text{for all $s\in S$}\}.
\end{eqnarray*} 
For \(\bla,\bmu\in\La^S(n)\) we define \(\bla+\bmu\) to be $\bnu\in\La^S(n)$ with \(\nu^{(s)} = \la^{(s)} + \mu^{(s)}\) for all \(s \in S\). 
For $\bla\in\La^S(n)$, we define its Young diagram to be $[\bla]:=\bigsqcup_{s\in S}[\la^{(s)}]$. We also set 
$$\|\bla\|:=(|\la^{(s)}|)_{s\in S}\in\Z_{\geq 0}^S.$$ 
For $d\in\Z_{\geq 0}$, we set 
\begin{align*}
\La^S(n,d)&:=\{\bla\in\La^S(n)\mid {\textstyle\sum_{s\in S}}|\la^{(s)}|=d\},
\\
\La_+^S(n,d)&:=\{\bla\in\La^S_+(n)\mid {\textstyle\sum_{s\in S}}|\la^{(s)}|=d\}.
\end{align*}  

In the special case $S=I=\{0,\dots,l\}$, we also write  $\bla=(\la^{(0)},\dots,\la^{(l)})$ instead of $\bla=(\la^{(i)})_{i\in I}\in\La^I(n)$. 
	For $i\in I$, and $\la\in\La(n,d)$, define 
	\begin{equation}\label{EIota}
		\biota_i(\la):=(0,\dots,0,\la,0,\dots,0)\in\La^I(n,d),
	\end{equation}
	with $\la$ in the $i$th position. 

	Let $\leq$ be a partial order on $S$. We have a partial order $\unlhd_S$ on the set $\Z_{\geq 0}^S$ with $(a_s)_{s\in S}\unlhd_S(b_s)_{s\in S}$ if and only if $\sum_{t\geq s}a_t\leq \sum_{t\geq s}b_t$ for all $s\in S$. Let $\unlhd$ be the usual {\em dominance partial order} on $\La(n,d)$, i.e. 
	$
	\la\unlhd \mu$ if and only if $\sum_{r=1}^s\la_r\leq \sum_{r=1}^s\mu_r$ for all $s=1,\dots,n$. 
	We have a partial order $\leq_S$ on $\La^S(n,d)$ defined as follows:
	$\bla\leq_S \bmu$ if and only if either $\|\bla\|\lhd_S\|\bmu\|$, 
	or $\|\bla\|=\|\bmu\|$ and  $\la^{(s)}\unlhd \mu^{(s)}$ for all $s\in S$.

\subsection{Tableaux} 
	\label{SSTab}
	Let $I,X,Y$ be heredity data on a $\k$-superalgebra $A$ as in \S\ref{SSBQHA}. We introduce {\em colored alphabets}   
	$
	\Alph_{X}:=[n]\times X \quad\text{and}\quad  \Alph_{X(i)}:=[n]\times X(i).
	$
	An element $(l,x)\in \Alph_{X}$ is often written as $l^x$. 
	If  $L=l^x\in \Alph_{X}$, we denote 
	$
	\letter(L):=l$ and $\col(L):=x.$
	For all $i\in I$, we fix arbitrary total orders `$<$' on the sets $\Alph_{X(i)}$ which satisfy $r^x< s^x$ if  $r< s$ (in the standard order on $[n]$). 
	All definitions of this subsection which involve $X$ have obvious analogues for $Y$, for example, we have the colored alphabets $\Alph_{Y}$ and $\Alph_{Y(i)}$.

	Let $\bla=(\la^{(0)},\dots,\la^{(l)})\in\La^I(n,d)$. Fix $i\in I$. A  {\em standard $X(i)$-colored $\la^{(i)}$-tableau} is a function
	$T:\Brackets{\la^{(i)}}\to \Alph_{X(i)}$ such that the following two conditions are satisfied:
\begin{enumerate}
\item[{\sf (R)}] $T(M) \leq T(N)$ whenever $M\prec N$ are nodes in the same row of $\Brackets{\la^{(i)}}$, and the equality is allowed only if $\col(T(M))\in X(i)_\0$.
\item[{\sf (C)}] $T(M) \leq T(N)$ whenever $M\prec N$ are nodes in the same column of $\Brackets{\la^{(i)}}$, and the equality is allowed only if $\col(T(M))\in X(i)_\1$. 
\end{enumerate}
We denote by 
$\Std^{X(i)}(\la^{(i)})$ 
the set of all standard 
$X(i)$-colored $\la^{(i)}$-tableaux. 
Recalling the idempotents $e_i\in X(i)\cap Y(i)$, the {\em initial $\la^{(i)}$-tableau} $T^{\la^{(i)}}$ is
	$
	T^{\la^{(i)}}: \Brackets{\la^{(i)}}\to \Alph_{X(i)}, \ (r,s)\mapsto r^{e_i}.
	$
	Note that $T^{\la^{(i)}}$ is in both $\Std^{X(i)}(\la^{(i)})$ and $\Std^{Y(i)}(\la^{(i)})$. 
	
	Let $T\in\Std^{X(i)}(\la^{(i)})$. Denote $d_i:=|\la^{(i)}|$. Reading the entries of $T$ along the rows from left to right starting from the first row, we get a sequence $l_1^{x_1}\cdots l_{d_i}^{x_{d_i}}\in \Alph_{X(i)}^{d_i}$. We denote $\bl^T:=l_1\cdots l_{d_i}$ and $\bx^T:=x_1\cdots x_{d_i}$.

	For a function $\T:\Brackets\bla\to \Alph_{X}$ and $i\in I$, we set  $T^{(i)}:=\T|_{\Brackets{\la^{(i)}}}$ to be the restriction of $\T$ to $\Brackets{\la^{(i)}}$. We write $\T=(T^{(0)},\dots,T^{(l)})$, keeping in mind that the restrictions $T^{(i)}$ determine $\T$ uniquely. 
	A {\em standard $X$-colored $\bla$-tableau} is a function
	$\T:\Brackets\bla\to \Alph_{X}$ such that $T^{(i)}\in\Std^{X(i)}(\la^{(i)})$ 
	for all $i\in I$. 
	We denote by $\Std^X(\bla)$ the set  of all standard $X$-colored $\bla$-tableaux. For example, we have the {\em initial $\bla$-tableau}\, 
$$\T^\bla=(T^{\la^{(0)}},\dots,T^{\la^{(l)}})\in \Std^X(\bla)\cap \Std^Y(\bla).$$ 
	For $\T\in\Std^{X}(\bla)$, we denote 
$$\bl^\T:=\bl^{T^{(0)}}\cdots \,\bl^{T^{(l)}}\in[n]^d,\quad \bx^\T:=\bx^{T^{(0)}}\cdots\, \bx^{T^{(l)}}\in X^d,\quad\text{and}\quad\bl^\bla:=\bl^{\T^\bla}.
$$ 
The sequence $\by^\T$ for $\T\in\Std^{Y}(\bla)$ is defined similarly to $\bx^\T$. 
	
	Let $\bla\in\La_+^I(n,d)$ and $\T\in\Std^X(\bla)$, with $\bl^\T=l_1\cdots l_d$ and $\bx^\T=x_1\cdots x_d$. Suppose that there exist $i_1,\dots,i_d\in I$ such that $e_{i_1}x_1=x_1,\,\dots,\,
	e_{i_d}x_d=x_d$. Recalling (\ref{EEps}) and (\ref{EIota}), we define the left weight of $\T$ to be
$$
	\balpha(\T):=\sum_{c=1}^d \biota_{i_c}(\eps_{l_c})\in \La^I(n,d). 
$$
	For $\bmu\in\La^I(n,d)$, we denote 
	\begin{equation}\label{EAl}
		\Std^X(\bla,\bmu):=\{\T\in \Std^X(\bla)\mid \balpha(\T)=\bmu\}.
	\end{equation}
	
\subsection{Symmetric functions}	
	Let $\Sym$ be the ring of symmetric functions over $\Z$ in infinitely many  variables $z_1,z_2,\dots$, see \cite{Mac}, with  basis consisting of Schur functions $s_\la\in\Sym$ for $\la\in\La_+$. 
	Recall that $\Sym$ is a Hopf algebra with coproduct 
	$${\tt\Delta}:\Sym\to \Sym\otimes\Sym,\ s_\la\mapsto \sum_{\mu,\nu\in\La_+}c^{\,\la}_{\mu,\nu}\, s_\mu\otimes s_\nu,$$
where  $c^{\,\la}_{\mu,\nu}$ are the Littlewood-Richardson coefficients, 
see \cite[\S I.5]{Mac}. 

For $n\in\Z_{>0}$, let $\Sym(n)=\Z[z_1,\dots,z_n]^{\Si_n}$ be the ring of symmetric polynomials in $z_1,\dots,z_n$. There is a canonical homomorphism $\rho_n:\Sym\to\Sym(n)$, see \cite[p.18]{Mac}. For $\la\in\La_+(n)$, let $s_\la(z_1,\dots,z_n):=\rho_n(s_\la)\in\Sym(n)$.

For a finite set $S$, we introduce $S$-fold tensor products $\Sym^S:=\Sym^{\otimes S}$ and $\Sym^S(n):=\Sym(n)^{\otimes S}$.  We have the canonical homomorphism   
	\begin{equation}\label{ERhoMult}
		\rho_n^S=\rho_n^{\otimes S}:\Sym^S\to\Sym^S(n).
\end{equation}
For $d\in\Z_{\geq 0}$, we denote by $\Sym^S(n,d)$ the degree $d$ component of $\Sym^S(n)$. 
	
Given $\bnu=(\nu^{(s)})_{s\in S}\in \La^S_+$, we have an element $$
s_\bnu:=\otimes_{t\in S} s_{\nu^{(t)}}\in \Sym^{S}.
$$ 
	If $\bnu\in \La^S_+(n)$, we set 
$$s_\bnu(z_1,\dots,z_n):=\rho_n^S(s_\bnu)\in \Sym^{S}(n).$$

	If $m=|S|$, iterating the coproduct (and using coassociativity and cocommutativity) we get the algebra homomorphism
\begin{equation}\label{EIteratedCoproduct}
{\tt\Delta}^{m-1}: \Sym\to \Sym^{S},
\end{equation}
(with ${\tt\Delta}^{0}$ interpreted as the identity map). 
	We introduce the iterated Littlewood-Richardson coefficients $c^\la_\bnu$ from
	\begin{equation}\label{EMLR}
		{\tt\Delta}^{m-1}(s_\la)=\sum_{\bnu\in \La^S_+}c^\la_\bnu s_\bnu.
	\end{equation}

\section{Generalized Schur algebras}
	\label{SSA}
	Throughout the section, we fix $n\in\Z_{> 0}$. 
We also fix a based quasi-hereditary superalgebra $A_R$ over $R$ with conforming heredity data $I,X,Y$. 
	
	\subsection{Definition}
	\label{SSDefT}
	Let $S$ be a set and $d\in\Z_{\geq 0}$. Recall that the symmetric group $\Si_d$ acts on $S^d$ by place permutations. 
	For $\bs,\bt\in S^d$, we write $\bs\sim\bt$ if $\bs\si=\bt$ for some $\si\in \Si_d$. If $S_1,\dots,S_m$ are sets, then $\Si_d$ acts on $S_1^d\times\dots\times S_m^d$ diagonally. We write 
	$(\bs_1,\dots,\bs_m)\sim (\bt_1,\dots,\bt_m)$ 
	if 
	$(\bs_1,\dots,\bs_m)\si=(\bt_1,\dots,\bt_m)$ for some $\si\in \Si_d$. 
	If $U\subseteq S_1^d\times\dots\times S_m^d$ is a $\Si_d$-invariant subset, we denote by $U/\Si_d$ a complete set of the $\Si_d$-orbit representatives in $U$ and we identify $U/\Si_d$ with the set of all $\Si_d$-orbits on $U$. 
	
	Let $H=H_\0\sqcup H_\1$ be a set of non-zero homogeneous elements of $A_R$. 
	Define $\Seq^H (n,d)$ to be the set of all triples 
	$$
	(\ba,\br, \bs) = ( a_1\cdots a_d,\, r_1\cdots r_d,\, s_1\cdots s_d ) \in  H^d\times[n]^d\times[n]^d
	$$
	such that for all $1\leq k\neq l\leq d$ we have 
	$(a_k,r_k,s_k)=(a_l,r_l, s_l)$ 
	only if $a_k\in H_\0$. Then $\Seq^H (n,d)\subseteq H^d\times[n]^d\times[n]^d$ is a $\Si_d$-invariant subset, so we can choose a set $\Seq^H (n,d)/\Si_d$ of $\Si_d$-orbit representatives and identify it with the set of all $\Si_d$-orbits on $\Seq^H (n,d)$ as in the previous paragraph. 

Sometimes we use a preferred choice of representatives for $\Seq^H (n,d)/\Si_d$ defined as follows. Fix a total order $<$ on 
	$H\times[n]\times[n]$. We have a lexicographic order on $\Seq^H(n,d)$: $(\ba,\br,  \bs)< (\ba',\br',  \bs')$ if and only if there exists $l\in[d]$ such that $(a_k,r_k,s_k)=(a_k',r_k',s_k')$ for all $k<l$ and $(a_l,r_l,s_l)<(a_l',r_l',s_l')$. Denote
	\begin{equation}\label{ESeq0}
		\Seq^H_0(n,d)=\{(\ba, \br, \bs) \in \Seq^H(n,d)\mid (\ba, \br, \bs)\leq (\ba, \br, \bs) \sigma\ \text{for all}\ \sigma \in \mathfrak{S}_d\}.
	\end{equation}

For $(\ba,\br, \bs) \in \Seq^H(n,d)$ and $\si\in\Si_d$, 
	we define
	\begin{align}
		\label{EAngle3}
		\lan\ba, \br,  \bs\ran
		&:=\sharp\{(k,l)\in[d]^2\mid k<l,\ a_k,a_l\in H_\1,\ (a_k,r_k,s_k)> (a_l,r_l, s_l)\}.
	\end{align}
	
Specializing to $H=B$, let $(\bb,\br,  \bs) \in \Seq^B(n,d)$. 
	For $b\in B$ and $r,s\in [n]$, we denote 
$$
[\bb,\br, \bs:b,r,s]:=\sharp\{k\in[d]\mid  (b_k,r_k,s_k)=(b,r,s)\},
$$ 
and, recalling (\ref{EBABC}), we set 
\begin{equation}\label{ECFactorial}
		[\bb,\br, \bs]^!_{\c} :=\prod_{ b\in B_\c,\, r,s\in [n]}[\bb,\br, \bs:b,r,s]!.
	\end{equation}

	Let $M_n(A_R)$ be the superalgebra of $n\times n$ matrices with entries in $A_R$. For $a\in A_R$, we denote by  
	$
	\xi_{r,s}^a\in M_n(A_R)
	$ 
	the matrix with $a$ in the position $(r,s)$ and zeros elsewhere. By definition, $|\xi_{r,s}^a|=|a|$. 
For each $d\in\Z_{\geq 0}$ we have a superalgebra structure on $M_n(A_R)^{\otimes d}$, and then on \(\bigoplus_{d\geq 0} M_n(A_R)^{\otimes d}\).

Recall from \cite[\S4.1]{EK1}, that \(\bigoplus_{d\geq 0} M_n(A_R)^{\otimes d}\) is a bisuperalgebra 
	with 
	the coproduct \(\coproduct\) defined by
	\begin{align*}
		\coproduct\,:\,\, M_n(A_R)^{\otimes d}\,\, &\to \,\,\bigoplus_{c=0}^d M_n(A_R)^{\otimes c} \otimes M_n(A_R)^{\otimes (d-c)}\\
		\xi_1 \otimes \cdots \otimes \xi_d\,\,&\mapsto\,\, \sum_{c=0}^d (\xi_1 \otimes \cdots \otimes \xi_c) \otimes (\xi_{c+1} \otimes \cdots \otimes \xi_d).
	\end{align*}
Moreover, recalling (\ref{EStarNotationGen}), \(\bigoplus_{d\geq 0} M_n(A_R)^{\otimes d}\) is also a bisuperalgebra with respect to $\coproduct$ and $*$, see \cite[Lemma 3.12]{EK1}.


	According to (\ref{ESiAct}) $\Si_d$ acts on $M_n(A_R)^{\otimes d}$ with superalgebra automorphisms, and using the notation (\ref{EGa}), we have the subsuperalgebra of invariants 
	$
	\Ga^d M_n(A_R)\subseteq M_n(A_R)^{\otimes d}.
	$ 
	For $(\ba,\br,\bs)\in\Seq^H(n,d)$, we have elements 
	\begin{equation*}\label{EXiDef}
		\xi_{\br,\bs}^\ba:= \sum_{(\bc,\bt,\bu)\sim(\ba,\br,\bs)} 
		(-1)^{\lan\ba,\br,\bs\ran+\lan\bc,\bt,\bu\ran}
		\xi_{t_1,u_1}^{c_1}\otimes\dots\otimes \xi_{t_d,u_d}^{c_d}
		\in \Ga^d M_n(A_R).
\end{equation*}
We have the following $R$-basis of $\Ga^d M_n(A_R)$:
\begin{equation}\label{EBasisS}
	\{\xi_{\br,\bs}^\bb \mid (\bb,\br,\bs)\in\Seq^B(n,d)/\Si_d\}. 
\end{equation}

For $(\bb,\br,\bs)\in\Seq^B(n,d)$, we also set 
$$
\eta^\bb_{\br,\bs}:=[\bb,\br,  \bs]^!_{\c}\, \xi^\bb_{\br,  \bs}, 
$$
and  
$$
T(n,d)_R=T^{A}(n,d)_R:=\spa_R\big\{\,\eta^\bb_{\br,\bs}\mid (\bb,\br,\bs)\in\Seq^B(n,d)\,\big\}\subseteq \Ga^d M_n(A_R).
$$ 
Let 
$$
T(n)_R := \bigoplus_{d\geq 0} T(n,d)_R.
$$

By \cite[Proposition 3.12, Lemma 3.10]{KMgreen2}, $T(n,d)_R$ is a unital $R$-subsuperalgebra of $M_n(A_R)^{\otimes d}$ with $R$-basis 
\begin{equation}\label{EBasisT}
	\big\{\,\eta^\bb_{\br,\bs}\mid (\bb,\br,\bs)\in\Seq^B(n,d)/\Si_d\,\big\}. 
\end{equation}
Moreover, by \cite[Corollary 3.24]{KMgreen2}, 
$T(n)_R$ is a sub-bisuperalgebra of\, $\bigoplus_{d\geq 0} M_n(A_R)^{\otimes d}$ (with respect to $\coproduct$ and the usual product). Moreover:

\begin{Lemma} \label{LNablaStar} {\rm \cite[Corollary 4.4]{KMgreen2}} 
	$T(n)_R$ is a sub-bisuperalgebra of $\bigoplus_{d\geq 0} M_n(A_R)^{\otimes d}$ with respect to the coproduct $\coproduct$ and the product $*$
\end{Lemma}


Extending scalars from $R$ to $\F$, we now define the $\F$-superalgebra 
$$T(n,d)_\F=T^A(n,d)_\F:=\F\otimes_R T(n,d)_R.$$
We denote $1_\F\otimes \eta^\bb_{\br,\bs}\in T(n,d)_\F$ again by $\eta^\bb_{\br,\bs}$, the map $\id_\F\otimes\, \coproduct$ again by $\coproduct$, etc. In fact, when working over the field, we will often drop the index and write simply 
\begin{equation}\label{ENoIndex}
T(n,d):=T(n,d)_\F.
\end{equation}

If $W_1$ is a $T(n,d_1)$-supermodule and $W_2$ is a $T(n,d_2)$-supermodule, we  consider $W_1\otimes W_2$ as a $T(n,d_1+d_2)$-supermodule via the coproduct $\coproduct$.

\subsection{Properties of product and coproduct}
\label{SSProdCoprod}
In this section we work over $R$. 
Define the structure constants $\kappa^b_{a,c}\in R$ of $A_R$ from
$
ac=\sum_{b\in B}\kappa^b_{a,c} b
$
for $a,c\in A_R$.
More generally, for 
$\bb=(b_1,\dots, b_d)\in B^d\quad \text{and}\quad \ba=(a_1,\dots, a_d),\,\bc=(c_1,\dots, c_d)\in A_R^d,$ 
we define 
$$
\ka^\bb_{\ba,\bc}:=\ka^{b_1}_{a_1,c_1}\cdots \ka^{b_d}_{a_d,c_d}\in R.
$$
Recall the notation (\ref{EAngle3}),  (\ref{EAngle2}). 
The following generalization of \cite[(2.3b)]{Green} follows from  \cite[(6.14)]{EK1}, cf. \cite[Proposition 3.6]{KMgreen2}.

\begin{Proposition} \label{CPR} 
	Let $(\ba,\bp,  \bq),\, (\bc,\bu,  \bv) \in \Seq^B(n,d)$.
	Then in $T(n,d)_R$ we have
	$$
	\eta^\ba_{\bp,  \bq}\, \eta^\bc_{\bu,  \bv}=\sum_{[\bb,\br,\bs]\in\Seq^B(n,d)/\Si_d} g_{\ba,\bp, \bq;\bc,  \bu,\bv}^{\bb,\br,\bs}\,  \eta^\bb_{\br,  \bs}
	$$
	where 
	$$
	g_{\ba,\bp, \bq;\bc,  \bu,\bv}^{\bb,\br,\bs}= 
	\frac{[\ba, \bp, \bq]^!_{\c} \cdot [\bc, \bu, \bv]^!_{\c}}{[\bb, \br, \bs]^!_{\c}}
	\sum_{\ba', \bc',\bt} 
	(-1)^{\lan\ba,\bp,  \bq\ran 
		+\lan\bc,\bu,  \bv\ran
		+ \lan\ba',\br,  \bt\ran
		+\lan\bc',\bt, \bs\ran+\lan\ba',\bc'\ran} \, 
	\kappa_{\ba',\bc'}^{\bb},
	$$
	the sum being over all 
	$\ba', \bc'\in B^d$ and $\bt\in[n]$  
	such that $(\ba',\br,  \bt) \sim (\ba,\bp,  \bq)$ and $(\bc',\bt,  \bs)\sim (\bc,\bu,\bv)$. 
\end{Proposition}

\begin{Lemma} \label{LSingleSep} {\rm \cite[Lemma 4.6]{KMgreen2}} Let $q\in\Z_{>0}$, $d_1,\dots,d_q\in\Z_{\geq 0}$ with $d_1+\dots+d_q=d$, and for $m=1,\dots,q$, we have 
	$
	(\bb^{m},\br^{m}, \bs^{m})
	\in\Seq^B(n,d_m)
	$ with $\bb^{m}=b^{m}_1\cdots b^{m}_{d_m},\ \br^{m}=r^{m}_1\cdots r^{m}_{d_m},\ 
	\bs^{m}=s^{m}_1\cdots s^{m}_{d_m}
	$. 
	If 
	$(b^{m}_t,r^{m}_t,s^{m}_t)\neq (b^{l}_u,r^{l}_u,s^{l}_u)$ for all $1\leq m\neq l\leq q$, $1\leq t\leq d_m$ and $1\leq u\leq d_l$, then  
	$$
	\eta^{\bb^1\cdots\bb^q}_{\br^1\cdots\br^q,\bs^1\cdots\bs^q}=\eta_{\br^{1},\bs^{1}}^{\bb^{1}}*\dots*\eta_{\br^{q},\bs^{q}}^{\bb^{q}}.
	$$
\end{Lemma}

To describe $\coproduct$ on basis elements, let $\Triple=(\bb,\br,\bs)\in\Seq^B_0(n,d)$. We write 
$
\eta_\Triple:=\eta^\bb_{\br,\bs}$ and $\Triple\si:=(\bb,\br,\bs)\si$ for $\si\in \Si_d$.  
We have that the stabilizer 
$
\Si_\Triple:=\{\si\in\Si_d\mid \Triple\si=\Triple\}
$
is a standard parabolic subgroup. Let ${}^\Triple\D$ be the set of the shortest coset representatives in $\Si_\Triple\backslash\Si_d$. We also set 
\begin{equation}\label{EC!}
	[\Triple]^!_{\c}:=[\bb,\br,\bs]^!_{\c}.
\end{equation}
If $d=d_1+d_2$, $\Triple^1=(\bb^1, \br^1, \bs^1)\in\Seq^B(n,d_1)$ and $\Triple^2=(\bb^2, \br^2, \bs^2)\in\Seq^B(n,d_2)$, we denote 
$
\Triple^1\Triple^2:=(\bb^1\bb^2, \br^1\br^2, \bs^1\bs^2)\in B^d\times[n]^d\times [n]^d.
$
Recall the notation (\ref{ESeq0}). For \(\Triple \in \Seq^B_0(n,d)\)  define 
\begin{equation}\label{ESplit}
	\textup{Spl}(\Triple):=\bigsqcup_{0\leq e\leq d}\big\{(\Triple^1, \Triple^2)\in \Seq^B_0(n,e) \times \Seq^B_0(n,d-e) \mid 
	\Triple^1\Triple^2\sim \Triple\big\}.
\end{equation}
For $(\Triple^1,\Triple^2) \in \textup{Spl}(\Triple)$, let \(\si^{\Triple}_{\Triple^1,\Triple^2}\) be the unique element of\, ${}^{\Triple}\mathscr{D}$ such that 
$
\Triple\si^{\Triple}_{\Triple^1,\Triple^2}  = \Triple^1\Triple^2.
$ 
Recalling the notation (\ref{EAngleSi}), we have:

\begin{Lemma}\label{coprodeta} {\rm \cite[Corollary 3.24]{KMgreen2}} 
	If \(\Triple=(\bb,\br,\bs) \in \Seq^B_0(n,d)\) then 
	\begin{align*}
		\coproduct(\eta_\Triple)=\sum_{(\Triple^1,\Triple^2) \in \textup{Spl}(\Triple)} 
		(-1)^{\langle \si^{\Triple}_{\Triple^1,\Triple^2}; \bb \rangle}
		{\small \frac{[\Triple]^!_{\c}}{[\Triple^1]^!_{\c}[\Triple^2]^!_{\c}}}
		\eta_{\Triple^1} \otimes \eta_{\Triple^2}.
	\end{align*}
\end{Lemma}

\subsection{Idempotents and characters}
Let $\la\in\La(n,d)$. Set 
$
\bl^\la:=1^{\la_1}\cdots n^{\la_n}.
$ 
For an idempotent $e\in A$ we have an idempotent 
$
\eta_\la^e:=\eta^{e^d}_{\bl^\la,\bl^\la}\in T(n,d)$. 
Let $e_0,\dots,e_l\in A$ be the standard idempotents. 
For each \(\bla=(\la^{(0)},\dots,\la^{(l)}) \in \La^I(n,d)\), we have the idempotent
$$
\eta_{\bla}:=\eta_{\la^{(0)}}^{e_0} * \cdots * \eta_{\la^{(l)}}^{e_l} \in T^{A}(n,d).
$$
The idempotents $\eta_{\bla}$ are orthogonal.



For $\la=(\la_1,\dots,\la_n)\in\La(n)$, define the monomial  $z^\la:=z_1^{\la_1}\cdots z_n^{\la_n}\in\Z[z_1,\dots,z_n].$ 
For $\bla\in\La^I(n)$, we now set 
$$z^\bla:=z^{\la^{(0)}}\otimes z^{\la^{(1)}}\otimes\dots\otimes z^{\la^{(l)}}\in \Z[z_1,\dots,z_n]^{\otimes I}.$$
Following \cite[\S5A]{KMgreen2}, see especially \cite[Lemma~5.9]{KMgreen2}, for a $T(n,d)$-module $V$, we define its {\em formal character}  
$$
\ch V := \sum_{\bmu\in\La^I(n,d)}(\dim\, \eta_\bmu V)z^\bmu\in \Sym^I(n,d).
$$
If $\sum_{i\in I}e_i=1_A$, then 
$1_{T^{A}(n,d)}=\sum_{\bla\in\La^{I}(n,d)}\eta_\bla,
$
but we do not need to assume this. So in general we might have $\sum_{\bmu\in\La^I(n,d)}\eta_\bmu V\subsetneq V$.

\begin{Lemma}\label{LChProd} \cite[Lemma 5.10]{KMgreen2}
If\, $W_1\in\mod{T(n,d_1)}$ and\, $W_2\in\mod{T(n,d_2)}$, then\, 
$
\ch(W_1 \otimes W_2) = \ch(W_1) \, \ch(W_2).
$
\end{Lemma}

The group $\Si_n$ acts on $\La(n)$ on the left via  
\begin{equation}\label{ESiNAct}
\si\la:=(\la_{\si^{-1}1},\dots , \la_{\si^{-1}n}).
\end{equation}
The group $\Si_n^I:= \prod_{i \in I} \Si_n$ acts on \(\La^I(n)\) via 
$
\bsi\bla:=(\si^{(0)}\la^{(0)},\dots , \si^{(l)}\la^{(l)}),
$
for $\bsi=(\si^{(0)},\dots,\si^{(l)})\in\Si_n^I$ and $\bla=(\la^{(0)},\dots,\la^{(l)})\in\La^I(n)$. For $a\in A$ and \(\sigma \in \Si_n\), let \(\xi_\sigma^a:=\sum_{r=1}^n\xi_{\si(r),r}^a\in M_n(A)\). 
For \(\bsi=(\si^{(0)},\dots,\si^{(l)}) \in \Si^I_n\), we set
\begin{equation}\label{EXiSi}
\xi_d(\bsi):=\sum_{d_0+\dots+d_l=d}  (\xi_{\sigma^{(0)}}^{e_0})^{\otimes d_0} * \cdots * (\xi_{\sigma^{(l)}}^{e_l})^{\otimes d_l}\in T^A(n,d).
\end{equation}

\begin{Lemma}\label{LWeyl} \cite[Lemmas 5.6,\,\,5.7]{KMgreen2}
For all $\bsi, \btau \in \Si^I_n$ and $\bla \in \La^I(n,d)$, we have $\xi_d(\bsi) \xi_d(\btau) = \xi_d(\bsi \btau)$ and $\xi_d(\bsi)\eta_{\bla}\xi_d(\bsi^{-1}) = \eta_{\bsi \bla}$.
\end{Lemma}

\begin{Corollary} \label{CActionOnWtSp}
For $\bsi \in \Si^I_n$, $\bla \in \La^I(n,d)$ and $V\in\mod{T(n,d)}$, we have  $\xi_d(\bsi)\,\eta_\bla V=\eta_{\bsi\bla}V$. 
\end{Corollary}


\begin{Lemma}\label{LDeWeyl} 
For $\bsi \in \Si^I_n$, we have\, $\coproduct(\xi_d(\bsi)) =\sum_{c=0}^d\xi_{c}(\bsi)\otimes \xi_{d-c}(\bsi)$.
\end{Lemma}
\begin{proof}
By definition, 
$$
\coproduct((\xi_{\sigma^{(i)}}^{e_i})^{\otimes d_i})=
\sum_{c_i=0}^{d_i}\big((\xi_{\sigma^{(i)}}^{e_i})^{\otimes c_i}\big)
\otimes \big((\xi_{\sigma^{(i)}}^{e_i})^{\otimes d_i-c_i}\big).
$$
By Lemma~\ref{LNablaStar}, 
we have 
$$
\coproduct(\xi_d(\bsi))=\sum_{d_0+\dots+d_l=d}  \coproduct\big((\xi_{\sigma^{(0)}}^{e_0})^{\otimes d_0}\big) * \cdots * \coproduct\big((\xi_{\sigma^{(l)}}^{e_l})^{\otimes d_l}\big),
$$
and the result follows. 
\end{proof}

For $N\geq n$, we set 
$E^N_n:=\sum_{r=1}^n\xi_{r,r}^1\in M_N(A)$ and consider the idempotent 
\begin{equation}\label{E130818}
\eta^N_n(d):=(E^N_n)^{\otimes d}
\in T^A(N,d). 
\end{equation}

\begin{Lemma} \label{LIdEasy} {\rm \cite[Lemma 5.15]{KMgreen2}} 
Let $N\geq n$. Then we have a  unital superalgebra isomorphism 
$$
T(n,d)\iso \eta^{N}_n(d)T(N,d)\eta^{N}_n(d),\ \eta^\bb_{\br,\bs}\mapsto \eta^\bb_{\br,\bs}$$
\end{Lemma}


\begin{Lemma}
\label{LTensTrunc}
{\rm \cite[Proposition 5.19]{KMgreen2}}
Let $d_1,d_2\in\Z_{\geq 0}$ with $d_1+d_2=d$, 
$n\leq N$,  $V_1\in\mod{T(N,d_1)}$ and $V_2\in\mod{T(N,d_2)}$. Then there is a functorial  isomorphism of $T(n,d)$-modules
$
\eta^N_n(d)(V_1\otimes V_2)\simeq (\eta^N_n(d_1) V_1)\otimes  (\eta^N_n(d_2)V_2).
$
\end{Lemma}

\subsection{Quasi-hereditary structure on $T(n,d)$}
\label{SSQHT}
Recall that $A$ is a based quasi-hereditary superalgebra with conforming heredity data $I,X,Y$. Throughout this subsection, we assume that $d\leq n$. 
Then, by \cite[Theorem 6.6]{KMgreen3}, $T(n,d)=T^A(n,d)$ is a based quasi-hereditary algebra. 

We now describe the heredity data $\La_+^I(n,d),\X(n,d),\Y(n,d)$ for $T(n,d)$ following \cite[\S6]{KMgreen3}. To start with, we have already defined the partially ordered set $\La_+^I(n,d)$ of $I$-multipartitions with partial order $\leq_I$, see \S\ref{SSComb}. For $\bla\in\La_+^I(n,d)$ the corresponding sets $\X(\bla)=\{\X_\Stab\mid \Stab\in\Std^X(\bla)\}$ and $\Y(\bla)=\{\Y_\T\mid \T\in\Std^Y(\bla)\}$ are labeled by the standard $X$-colored and $Y$-colored $\bla$-tableaux,  respectively. Recalling the notation $\bx^\Stab$, $\bl^\Stab$, etc. 
from \S\ref{SSTab}, 
the elements $\X_\Stab$ and $\Y_\T$ are defined as follows. 
$$
\X_\Stab:=\eta^{\bx^\Stab}_{\bl^\Stab,\bl^\bla},\quad
\Y_\T:=\eta^{\by^\T}_{\bl^\bla,\bl^\T}. 
$$ 
For any $\bla\in\La_+^I(n,d)$, we have $\X_{\T^\bla}=\Y_{\T^\bla}=\eta_\bla$, so $\X(\bla)\cap\Y(\bla)=\{\eta_\bla\}$, and $\{\eta_\bla\mid\bla\in \La_+^I(n,d)\}$ are the standard idempotents of the heredity data. 

Let $\bla\in\La_+^I(n,d)$. The standard module $\De(\bla)$ has basis 
\begin{equation}\label{EBasisDe}
\{v_\T:=\X_\T v_\bla\mid \T\in\Std^X(\bla)\},
\end{equation} 
where $v_\bla$ is the (unique up to scalar) vector of weight $\bla$ in $\De(\bla)$. Moreover, if $\T\in \Std^X(\bla,\bmu)$ for some $\bmu\in\La^I(n,d)$, see (\ref{EAl}), then 
\begin{equation}\label{EWtvT}
v_\T\in\eta_{\bmu}\De(\bla).
\end{equation}

Corollary~\ref{CActionOnWtSp} immediately implies:  

\begin{Lemma} \label{LStab}
For $\bsi \in \Si^I_n$ and $\bla \in \La^I_+(n,d)$ 
such that $\bsi\bla=\bla$, we have  $\xi_d(\bsi)v_\bla =\pm v_\bla$. 
\end{Lemma}

If follows from \cite[Theorem 6.6]{KMgreen3} that the formal character of the standard module $\De(\bla)$ is of the form
\begin{equation}\label{ETopTerm}
\ch \De(\bla)=z^\bla+\sum_{\bmu<_I\bla}c_\bmu z^\bmu.
\end{equation}
This implies:

\begin{Lemma} \label{LChInd} 
The formal characters $\{\ch\De(\bla)\mid \bla\in\La_+^I(n,d)\}$ are linearly independent. In particular:
\begin{enumerate}
	\item[{\rm (i)}] if $V\in\mod{T(n,d)}$ has a standard filtration and $\ch V=\sum_{\bla\in\La^I(n,d)}m_\bla\ch\De(\bla)$ then every $\De(\bla)$ appears as a subquotient of the filtration exactly $m_\bla$ times. 
	\item[{\rm (ii)}] if $\bla\in\La_+^I(n,d)$, $\bmu\in\La_+^I(n,c)$ and $\De(\bla)\otimes\De(\bmu)$ has a standard filtration, then $\De(\bla+\bmu)$ appears in this filtration once and all other subquotients $\De(\bnu)$ of the filtration satisfy $\bnu<_I\bla+\bnu$. 
\end{enumerate} 
\end{Lemma}

\begin{Lemma} \label{LYRaising} 
Let $\bla\in\La^I_+(n,d)$, $\br,\bs\in[n]^d$ and $y_1,\dots,y_d\in Y$ with at least one $y_r\not\in X$. 
Suppose that $v\in\eta_\bnu\De(\bla)$ for some $\bnu\in\La^I(n,d)$ with $\|\bnu\|=\|\bla\|$. Then $\eta^{y_1\cdots y_d}_{\br,\bs}v=0$.
\end{Lemma}
\begin{proof}
Suppose $\eta^{y_1\cdots y_d}_{\br,\bs}v\neq 0$. Then $\eta^{y_1\cdots y_d}_{\br,\bs}\eta_\bnu=\eta^{y_1\cdots y_d}_{\br,\bs}$. So there exist $i_1,\dots,i_d\in I$ such that $y_1e_{i_1}=y_1,\,\dots,\,y_de_{i_d}=y_d$ and for all $i\in I$ we have $\sharp\{k\mid i_k=i\}=|\nu^{(i)}|=|\la^{(i)}|$. On the other hand, there exist $j_1,\dots,j_d$ such that $e_{j_1}y_1=y_1,\,\dots,\,e_{j_d}y_d=y_d$. By Lemma~\ref{idemactionNew}, $j_1\geq i_1,\dots,j_d\geq i_d$, and by the assumption that at least one $y_r\not\in X$, we have that at least one $j_r>i_r$. So $\eta^{y_1\cdots y_d}_{\br,\bs}v \in \eta_\bmu \De(\la)$ for $\bmu$ satisfying $\|\bmu\|\rhd_I\|\bla\|$, hence $\bmu>_I\bla$, which contradicts (\ref{ETopTerm}). 
\end{proof}

Recalling (\ref{EIota}), we have 

\begin{Lemma} \label{LDeCol} {\rm \cite[Theorem 6.17(i)]{KMgreen3}}
Let $\bla=(\la^{(i)})_{i\in I}\in\La_+^I(n,d)$. Then $$\De(\bla)\simeq\bigotimes_{i\in I}\De(\biota_i(\la^{(i)})).$$ 
\end{Lemma}

\subsection{Character formula}
\label{SSChF}
Throughout the subsection we continue to assume that $d\leq n$. 
Set 
$${}_jX(i):=\{x\in X(i)\mid e_jx=x\}\qquad(i,j\in I).$$ 
Note that ${}_j X(i)\neq \varnothing$ only if $j\leq i$. 
For $\bnu=(\bnu^{(x)})_{x\in X(i)}\in\La_+^{X(i)}$ and $j\in I$, we define 
$$
{}_j\bnu=(\bnu^{(x)})_{x\in {}_jX(i)}\in\La_+^{{}_jX(i)}.
$$

Fix $i\in I$ until the end of the subsection. We define an algebra homomorphism 
$$
\chi:\Sym^{X(i)}\to\Sym^I,\ \bigotimes_{x\in X(i)}f_x\mapsto \bigotimes_{j\in I}\prod_{x\in {}_jX(i)}f_x,
$$
cf. \cite[(7.41)]{KMgreen3}. By the Littlewood-Richardson rule, for $\bnu \in \La_+^{X(i)}$, we have 
\begin{equation}\label{EChiLR}
\chi \left( \bs_{\bnu} \right) = \sum_{\bga \in \La_+^I} \prod_{j \in I} c_{{}_j\bnu}^{\ga^{(j)}} \bs_{\bga}. 
\end{equation}

For a multipartition $\bnu\in \La^{X(i)}_+(n,d)$, we define 
its {\em superconjugate}  multipartition 
$$\bnu^\tr:=(\tilde\nu^{(x)})_{x\in X(i)},$$ 
where $\tilde\nu^{(x)}:=\nu^{(x)}$ if $x$ is even and $\tilde\nu^{(x)}$ is the conjugate partition $(\nu^{(x)})'$ if $x$ is odd.
Using \cite[(2.7),(3.8)]{Mac}, we have the algebra homomorphism 
$$
\tr:\Sym^{X(i)}\to \Sym^{X(i)}, s_{\bnu}\mapsto s_{\bnu^\tr}.
$$

Let $t:=|X(i)|$. By choosing a total order on $X(i)$ we will identify 
$\La_+^{X(i)}$ with $\La_+^t$, 
$\Sym^{X(i)}$ with $\Sym^{\otimes t}$, etc. In particular, we have a well-defined map $\tr:\Sym^{\otimes t}\to \Sym^{\otimes t}$. 
Recalling (\ref{ERhoMult}) and (\ref{EIteratedCoproduct}), we now have:

\begin{Theorem} \label{TCharacter} 
Let $d\in\Z_{\geq 0}$, $n\in\Z_{>0}$, with $d\leq n$, and 
$\la\in\La_+(n,d)\subseteq \La_+$ and $i\in I$. Then 
$$
\ch \De(\biota_i(\la))= \rho_n^I\circ \chi\circ \tr\circ {\tt\Delta}^{t-1}(s_\la).
$$
\end{Theorem}
\begin{proof}
By \cite[Proposition 7.45]{KMgreen3}, we have
\begin{align*}
	\ch\De(\biota_i(\la))&=\sum_{\bga\in\La^I_+(n)}\, \sum_{\bnu\in\La_+(n)^{t}}
	c^{\la}_{\bnu^{\tr}}
	\bigg(\prod_{j\in I}c^{\ga^{(j)}}_{\,{}_{j}\bnu}\bigg)
	s_{\bga}(z_1,\dots,z_n)
	\\
	&=
	\rho_n^I\Bigg(\sum_{\bga\in\La^I_+}\, \sum_{\bnu\in\La_+^{t}}
	c^{\la}_{\bnu^{\tr}}
	\bigg(\prod_{j\in I}c^{\ga^{(j)}}_{\,{}_{j}\bnu}\bigg)
	s_{\bga}\Bigg)
	\\
	&=
	\rho_n^I\circ \chi \Bigg(\sum_{\bnu\in\La_+^{t}}
	c^{\la}_{\bnu^{\tr}}
	s_{\bnu}\Bigg)
	\\
	&=
	\rho_n^I\circ \chi \circ \tr \Bigg(\sum_{\bnu\in\La_+^{t}}
	c^{\la}_{\bnu}
	s_{\bnu}\Bigg)
	\\
	&=
	\rho_n^I\circ \chi \circ \tr \circ {\tt\Delta}^{t-1}(s_\la),
\end{align*}
where we have used (\ref{EChiLR}) for the third equality and (\ref{EMLR}) for the last equality.
\end{proof}

\begin{Theorem} \label{TProdCharacter} 
Let $\la\in\La_+(n,d)$, $\mu\in \La_+(n,e)$ and $i\in I$. If $d+e\leq n$ then 
$$
\ch\big(\De(\biota_i(\la))\otimes\De(\biota_i(\mu))\big)=\sum_{\nu\in\La_+(n,d+e)}c^{\,\nu}_{\la,\mu}\ch\De(\biota_i(\nu)).
$$
\end{Theorem}
\begin{proof}
By Lemma~\ref{LChProd}, Theorem~\ref{TCharacter} and the Littlewood-Richardson rule, we have
\begin{align*}
	\ch\big(\De(\biota_i(\la))\otimes\De(\biota_i(\mu))\big)&=\big(\ch\De(\biota_i(\la))\big)\big(\ch\De(\biota_i(\mu))\big)
	\\
	&=\big(\rho_n^I\circ \chi\circ \tr\circ {\tt\Delta}^{t-1}(s_\la)\big)\,
	\big(\rho_n^I\circ \chi\circ \tr\circ {\tt\Delta}^{t-1}(s_\mu)\big)
	\\
	&=\rho_n^I\circ \chi\circ \tr\circ {\tt\Delta}^{t-1}(s_\la s_\mu)
	\\
	&=\rho_n^I\circ \chi\circ \tr\circ {\tt\Delta}^{t-1}\Bigg(\sum_{\nu\in\La_+(n,d+e)}c^{\,\nu}_{\la,\mu}s_\nu\Bigg)
	\\&
	=\sum_{\nu\in\La_+(n,d+e)}c^{\,\nu}_{\la,\mu}\ch\De(\biota_i(\nu)),
\end{align*}
as required.
\end{proof}

For $\bla=(\la^{(j)})_{j\in I}\in\La^I_+(n,d)$, $\bmu=(\mu^{(j)})_{j\in I}\in\La^I_+(n,e)$ and $\bnu=(\nu^{(j)})_{j\in I}\in\La_+^I(n,d+e)$ we define
\begin{equation}\label{EBoldLR}
c^{\,\bnu}_{\bla,\bmu}:=\prod_{j\in I} c^{\,\nu^{(j)}}_{\la^{(j)}\hspace{-.6mm},\mu^{(j)}}.
\end{equation}

\begin{Corollary}\label{DeChTens}
Let $\bla\in\La^I_+(n,d)$ and $\bmu\in\La^I_+(n,e)$. If $d+e\leq n$ then 
$$
\ch(\De(\bla)\otimes\De(\bmu))=\sum_{\bnu\in\La_+^I(n,d+e)}c^{\,\bnu}_{\bla,\bmu}\ch\De(\bnu).
$$
\end{Corollary}
\begin{proof}
This follows from Theorem~\ref{TProdCharacter} and Lemma~\ref{LDeCol}.  
\end{proof}

\section{Tensor products of standard modules have standard filtrations}
\label{SMain}
We again fix a based quasi-hereditary superalgebra $A_R$ over $R$ with conforming heredity data $I,X,Y$. Recalling the convention (\ref{ENoIndex}), we have the ($\F$-)superalgebra $T(n,d)=T^A(n,d)$. Under the assumption $d\leq n$, 
this superalgebra is based quasi-hereditary with heredity data $\La_+^I(n,d),\X(n,d),\Y(n,d)$, see \S\ref{SSQHT}. 

\subsection{Reduction} We begin to prove Main Theorem by 
reducing to the case of `one color' and `fundamental dominant weights', cf. \cite[(3.5)]{Wang}, \cite[Proposition 3.5.4(i)]{DonkinFilt}. For integer $0\leq c\leq n$, recalling (\ref{EEps}), we have 
$$\om_c:=\eps_1+\dots+\eps_c\in\La_+(n,c).$$

\begin{Proposition} \label{PRed} 
Suppose that for all $n\in\Z_{>0}$, $d,c\in\Z_{\geq 0}$ with $d+c\leq n$, $\la\in\La_+(n,d)$, and $i\in I$, the tensor product $\De(\biota_i(\la))\otimes\De(\biota_i(\om_c))$ has a standard filtration. Then for all $n\in \Z_{>0}$, $d,c\in\Z_{\geq 0}$ with $d+c\leq n$, $\bla\in\La^I_+(n,d)$, $\bmu\in\La_+^I(n,c)$, the tensor product $\De(\bla)\otimes\De(\bmu)$ has a standard filtration. 
\end{Proposition}
\begin{proof}
We apply induction on the total degree $d+c$, the base case $d+c=0$ being trivial, since $T(n,0)\cong \F$. Let $d+c>0$. 
Take $\bla=(\la^{(0)},\dots,\la^{(l)})\in\La^I_+(n,d)$ and $\bmu=(\mu^{(0)},\dots,\mu^{(l)})\in \La^I_+(n,c)$. For all $i\in I$, set $d_i=|\la^{(i)}|$ and $c_i:=|\mu^{(i)}|$. 
By Theorem~\ref{LDeCol}, we have 
\begin{equation}\label{E020920}
	\De(\bla)\otimes\De(\bmu)\simeq \bigotimes_{i\in I}\Big(\De(\biota_i(\la^{(i)}))\otimes \De(\biota_i(\mu^{(i)}))\Big).
\end{equation}

Suppose there exist distinct $j,k\in I$ with $d_j,d_k>0$. Then  $d_i<d$ for all $i\in I$. By the inductive assumption, for all $i\in I$, we then have that $\De(\biota_i(\la^{(i)}))\otimes \De(\biota_i(\mu^{(i)}))$ has a standard filtration. It follows from Lemma~\ref{LChInd} and Theorem~\ref{TProdCharacter} that in this filtration only subquotients of the form $\De(\biota_i(\nu^{(i)}))$ with $\nu^{(i)}\in\La_+(n,d_i+c_i)$ appear. Hence by Theorem~\ref{LDeCol}, the right hand side of (\ref{E020920}) has a filtration with subquotients of the form 
$\bigotimes_{i\in I}\De(\biota_i(\nu^{(i)}))\simeq\De(\bnu)$. Thus we may assume that there exists a unique $i$ with $d_i=d$ and $d_k=0$ for all $k\neq i$, i.e. $\bla=\biota_i(\la)$ for some $i\in I$ and $\la\in\La_+(n,d)$. Similarly we may assume that $\bmu=\biota_j(\mu)$ for some $j\in I$ and $\mu\in\La_+(n,c)$. Moreover, we may assume that $j=i$ since otherwise 
$\De(\bla)\otimes \De(\bmu)=\De(\biota_i(\la))\otimes \De(\biota_j(\mu))
$ is a standard module, thanks to Theorem~\ref{LDeCol}. 

We now also apply induction on the dominance order on $\mu$. If $\mu$ is minimal in the dominance order, then $\mu=\om_c$ and  we are done by assumption. Otherwise, we can write $\mu=\ga+\om_{r}$ for $\ga\in\La_+(n,s)$ with $0<s,r<c$.  By the inductive assumption on the degree, we have that $\De(\biota_i(\ga))\otimes \De(\biota_i(\om_r))$ has a standard filtration. By Lemma~\ref{LChInd} and Theorem~\ref{TProdCharacter}, in this filtration $\De(\biota_i(\mu))$ appears once and other standard subquotients are of the form $\De(\biota_i(\nu))$ with $\nu\lhd\mu$. 
By \cite[Proposition A2.2(i)]{DonkinQS}, there is a short exact sequence 
$$0\to\De(\biota_i(\mu))\to \De(\biota_i(\ga))\otimes \De(\biota_i(\om_r))\to Q\to 0,$$ where $Q$ has a standard filtration with subquotients of the form $\De(\biota_i(\nu))$ with $\nu\lhd\mu$. Tensoring with $\De(\biota_i(\la))$ we get a short exact sequence
$$0\to\De(\biota_i(\la))\otimes \De(\biota_i(\mu))\to \De(\biota_i(\la))\otimes\De(\biota_i(\ga))\otimes \De(\biota_i(\om_r))\to \De(\biota_i(\la))\otimes Q\to 0.$$
By induction on the dominance order, $\De(\biota_i(\la))\otimes Q$ has a standard filtration. By induction on the degree, using  Lemma~\ref{LChInd} and Theorem~\ref{TProdCharacter}, we have that $\De(\biota_i(\la))\otimes\De(\biota_i(\ga))$ has a standard filtration with subquotients of the form $\De(\biota_i(\ka))$, hence by assumption, the middle term has a standard filtration. By \cite[Proposition A2.2(v)]{DonkinQS}, the first term has a standard filtration. 
\end{proof}

\subsection{The filtration}\label{SecFilt}
In view of Proposition~\ref{PRed}, we now fix $i\in I$, $\la\in\La_+(n,d)$,  $c\in\Z_{>0}$ such that $d+c\leq n$, and set 
$$\bla:=\biota_i(\la),\quad \bmu:=\biota_i(\om_c).
$$ 
We have highest weight vectors $v_\bla\in\De(\bla)$ and $v_\bmu\in\De(\bmu)$. 

Recalling the action of $\Si_n$ on $\La(n)$ from (\ref{ESiNAct}), we denote
$$
\Si_\la:=\{\si\in\Si_n\mid \si\la=\la\}
$$
If 
$\la=(l_1^{a_1},\dots,l_k^{a_k})$ for $l_1>\dots>l_k\geq 0$ and $a_1,\dots,a_k>0$ with $a_1+\dots+a_k=n$ then $\Si_\la=\Si_{a_1}\times\dots\times\Si_{a_k}$.

Let $\Om:=\{P\subseteq [n]\mid |P|=c\}$. 
The group $\Si_n$ acts on $\Om$ via $\si P=\{\si p_1,\dots,\si p_c\}$ for $P=\{p_1,\dots,p_c\}\in\Om$ and $\si\in\Si_n$. 
Denote 
$$\eps_P:=\eps_{p_1}+\dots+\eps_{p_c}\in\La(n,c).$$ 
Note that $\si(\eps_P)=\eps_{\si P}$ for all $\si\in\Si_n$ and $P\in \Om$. We denote 
$$
\Om_\la:=\{P\in\Om\mid \la+\eps_P\in\La_+(n,d+c)\}.
$$

Given $P=\{p_1,\dots,p_c\}$ and $Q=\{q_1,\dots,q_c\}$ in $\Om$, with $1\leq p_1<\dots<p_c\leq n$ and $1\leq q_1<\dots<q_c\leq n$, we write $P<Q$ if and only if $(p_1,\dots,p_c)<(q_1,\dots,q_c)$ lexicographically. 
This yields the total order on $\Om$. 
Let $\Om_\la=\{P_1,P_2,\dots,P_t\}$ with  
$$P_1=\{1,2,\dots,c\}<P_2<\dots<P_t.$$ 

The following is easy to see:

\begin{Lemma} \label{LOrb} 
Let $1\leq r\leq t$. Then 
\begin{enumerate}
	\item[{\rm (i)}] $P_r$ is the minimal element of the orbit $\Si_\la\cdot P_r$;
	\item[{\rm (ii)}] 
	$\Om=\bigsqcup_{r=1}^t \Si_\la\cdot P_r$.
\end{enumerate}
\end{Lemma}
\begin{proof}
Write $\la=(l_1^{a_1},\dots,l_k^{a_k})$ for $l_1>\dots>l_k\geq 0$ and $a_1,\dots,a_k>0$ with $a_1+\dots+a_k=n$, so that  $\Si_\la=\Si_{a_1}\times\dots\times\Si_{a_k}$ and  $A_1:=[1,a_1],A_2:=[a_1+1,a_1+a_2],\dots,A_k:=[n-a_k+1,n]$ are the orbits of $\Si_\la$ on $[n]$. Now $P,Q\in\Om$ are in the same $\Si_\la$-orbit if and only if $|P\cap A_s|=|Q\cap A_s|$ for all $s=1,\dots,k$, and it is clear that each orbit has a unique element from $\Om_\la$ which is the lexicographically minimal element of the orbit.  
\end{proof}

Let $P=\{p_1,\dots,p_c\}\in\Om$ with $p_1<\dots<p_c$. There is a unique tableau  $\T^P\in\Std^X(\bmu)$ with $\bl^{\T^P}=p_1\cdots p_c$ and $\bx^{\T^P}=e_i^c$. We denote the corresponding standard basis vector 
$$
w_P:=v_{\T^P}=\eta^{e_i^c}_{p_1\cdots p_c,12\cdots c}\,v_\bmu\in\De(\bmu),
$$
see (\ref{EBasisDe}). Note that the vectors $w_P$ do not exhaust the standard basis of $\De(\bmu)$.

\begin{Lemma} \label{LOneDim} 
Let $\nu\in\La(n,c)$. If $\eta_{\biota_i(\nu)}\De(\bmu)\neq 0$, then $\nu$ is of the form $\eps_P$ and $w_P$ spans $\eta_{\biota_i(\nu)}\De(\bmu)$.
\end{Lemma}
\begin{proof}
By (\ref{EBasisDe}), (\ref{EWtvT}), the weight space $\eta_{\biota_i(\nu)}\De(\bmu)\neq 0$ is spanned by the basis elements $v_\T$ such that $\T\in\Std^X(\bmu,\biota_i(\nu))$. As $\bmu=\biota_i(\om_c)$, we 
deduce, using the property (c) of Definition~\ref{DCC}, that $\T=\T_P$ for some $P\in\Om$, i.e. $v_\T=w_P$. 
\end{proof}

For $\si\in\Si_n$ denote 
$$
\biota_i(\si):=(1,\dots,1,\si,1,\dots,1)\in\Si_n^I,
$$
with $\si$ in the $i$th position. 
Recalling (\ref{EXiSi}), we have an element 
$$
\xi_c(\biota_i(\si))=(\xi_{\sigma}^{e_i})^{\otimes c}\in T^A(n,c).
$$

\begin{Lemma} \label{L030920} 
Let $P=\{p_1,\dots,p_c\}\in\Om$ with $p_1<\dots< p_c$, and $\si\in\Si_n$ such that $\si p_1<\dots< \si p_c$. Then  
$\xi_c(\biota_i(\si))w_P=w_{\si P}$.
\end{Lemma}
\begin{proof}
By definition of $\xi_{\sigma}^{e_i}$, we have in $T(n,c)$:
$$
(\xi_{\sigma}^{e_i})^{\otimes c} \,\eta^{e_i^c}_{p_1\cdots p_c,12\cdots c}
=
\eta^{e_i^c}_{\si p_1\cdots \si p_c,12\cdots c}.
$$
So 
\begin{align*}
	\xi_c(\biota_i(\si))w_P=
	(\xi_{\sigma}^{e_i})^{\otimes c} \,\eta^{e_i^c}_{p_1\cdots p_c,12\cdots c}\,v_\bmu
	=
	\eta^{e_i^c}_{\si p_1\cdots \si p_c,12\cdots c}\,v_\bmu=w_{\si P},
\end{align*}
as required. 
\end{proof}

\begin{Corollary} \label{CGenMu} 
Let $P\in \Om$. Then $T(n,c)w_P=\De(\bmu)$. 
\end{Corollary}
\begin{proof}
Take $\si\in\Si_n$ with $\si(p_a)=a$ for $a=1,\dots,c$. By Lemma~\ref{L030920}, we have $\xi_c(\biota_i(\si))w_P=w_{\{1,\dots,c\}}=v_\bmu$, and the result follows since $T(n,c)v_\bmu=\De(\bmu)$. 
\end{proof}

\begin{Lemma} \label{Span}
We have $T(n,d+c)(v_\bla\otimes w_{P_t})=\De(\bla)\otimes\De(\bmu)$. 
\end{Lemma}
\begin{proof}
Let $h$ be maximal with $\la_h>0$, so $\bl^\bla= 1^{\la_1}\cdots h^{\la_h}$. Since $n\geq d+c$, we have  $P_t=\{h+1,\dots,h+c\}$. 

By (\ref{EBasisDe}), $\De(\bla)$ is spanned by elements of the form $\eta_{\br,\bl^\bla}^\bx v_\bla$ for $(\bx,\br,\bl^\bla)\in\Seq^X(n,d)$. Let $\Triple'\in \Seq_0^X(n,d)$ with $\Triple'\sim (\bx,\br,\bl^\bla)$, see (\ref{ESeq0}). Then 
\begin{equation}\label{EFact1}
	[\Triple']^!_\c=1, 
\end{equation}
since $x=xe_i\in B_\a$ for all $x\in X(i)$, see (\ref{EBABC}) and (\ref{EC!}).

On the other hand, for $(\bb,\bt,\bu)\in\Seq^B(n,c)$, we have that $\eta_{\bt,\bu}^\bb w_{P_t}=0$ unless $\eta_{\bt,\bu}^\bb\eta_{\biota_i(\eps_{P_t})}=\eta_{\bt,\bu}^\bb$, and, in view of Corollary~\ref{CGenMu}, $\De(\bmu)$ is spanned by all $\eta_{\bt,\bu}^\bb w_{P_t}$ with $\bu\sim (h+1)\cdots(h+c)$. 

Let  $(\bx,\br,\bl^\bla)\in\Seq^X(n,d)$ and $(\bb,\bt,\bu)\in\Seq^B(n,c)$ satisfy $\bu\sim (h+1)\cdots(h+c)$, and $\Triple\in \Seq^B_0(n,d+c)$ be the initial triple with $\Triple\sim(\bx\bb,\br\bt,\bl^\bla\bu)$. 
Let $(\Triple^1,\Triple^2)\in \textup{Spl}(\Triple)$ with $\Triple^1\in\Seq^B_0(n,d)$ and $\Triple^2\in\Seq^B_0(n,c)$. 
Suppose $\Triple^1=(\ba,\bv,\bs)\not\sim (\bx,\br,\bl^\bla)$. Since $\bl^\bla=
1^{\la_1}\cdots h^{\la_h}$ and $\bu\sim (h+1)\cdots(h+c)$, 
we necessarily have that  $s_k\in\{h+1,\dots,h+c\}$ for some $1\leq k\leq d$. Hence $\eta_{\Triple^1}v_\bla=0$. 
Now, by Lemma~\ref{coprodeta} and (\ref{EFact1}), 
$$
\eta_{\br\bt,\bl^\bla\bu}^{\bx\bb}(v_\bla\otimes w_{P_t})=(\eta_{\br,\bl^\bla}^\bx v_\bla)\otimes (\eta^\bb_{\bt,\bu}w_{P_t}),
$$
which implies the lemma.
\end{proof}

For $r=0,1,\dots,t$, we denote 
$$
M_r:=T(n,d)\lan v_{\bla} \otimes w_{P_s}\mid 1\leq s\leq r\ran\subseteq \De(\bla)\otimes\De(\bmu).
$$
In view of Lemma~\ref{Span}, we have a filtration 
\begin{equation}\label{EMFilt}
0=M_0\subseteq M_1\subseteq\dots\subseteq M_t=\De(\bla)\otimes\De(\bmu).
\end{equation}
Our goal is to show that 
$M_r/M_{r-1}\simeq\De(\biota_i(\la+\eps_{P_r}))$ for all $r=1,\dots,t$, to get the required standard filtration of $\De(\bla)\otimes\De(\bmu)$.

\begin{Lemma} \label{LLevel} 
If $1\leq r\leq t$ and $P \in \Si_\la\cdot P_r$, then $v_{\bla} \otimes w_{P} \in M_r$.
\end{Lemma}
\begin{proof}
Write $P=\{p_1,\dots,p_c\}$ with $p_1<\dots<p_c$.  
Let $\si\in\Si_\la$ be such that $\si P_r=P$ and  $\si p_1<\dots <\si p_c$. 
Note using Lemmas~\ref{LStab}, \ref{L030920} and  \ref{LDeWeyl} that 
$$
v_\bla\otimes w_P=
\pm (\xi_{d}(\biota_i(\si))v_\bla)\otimes (\xi_{c}(\biota_i(\si))w_{P_r})=
\pm
\xi_{d+c}(\biota_i(\si))(v_{\bla} \otimes w_{P_r})\in M_r,$$
as required.
\end{proof}

\begin{Lemma} \label{HiWtQuo}
Let 
$1\leq r\leq t$ and $E\in\Y(n,d+c)$. If $E(v_\bla\otimes w_{P_r})\not\in M_{r-1}$ then $E=\eta_{\biota_i(\la+\eps_{P_r})}$. In particular, $M_r/M_{r-1}$ is a highest weight module of weight $\biota_i(\la+\eps_{P_r})$. 
\end{Lemma}
\begin{proof}
Suppose 
$E(v_\bla\otimes w_{P_r})\not\in M_{r-1}$. Write   $E=\Y_\T:=\eta^{\by^\T}_{\bl^\bnu,\bl^\T}$, with $\T\in\Std^Y(\bnu)$ for some $\bnu\in\La^I_+(n,d+c)$.
By Lemma~\ref{coprodeta}, $\coproduct(\eta^{\by^\T}_{\bl^\bnu,\bl^\T})$ is a linear combination of elements of the form $\eta_{\br,\bs}^\by\otimes \eta_{\br',\bs'}^{\by'}$ such that $\by\by'\sim\by^T$. By Lemma~\ref{LYRaising}, $\eta_{\br,\bs}^\by v_\bla\neq 0$ only if $\by=e_i^d$, and 
$\eta_{\br',\bs'}^{\by'} w_{P_r}\neq 0$ only if $\by'=e_i^c$. 
We conclude that 
$\by^\T=e_{i}^{d+c}$, and so 
$E$ can be written in the form $E=\eta^{e_i^{d+c}}_{\br,\bs}$ with $r_k\leq s_k$ for all $k$. 

If $r_k<s_k$ for some $k$, then in view of Lemma~\ref{coprodeta} and Lemma~\ref{LOneDim},  $E(v_\bla\otimes w_{P_r})$ is a multiple of $v_\bla\otimes w_P$ for some $P<P_r$. By Lemma~\ref{LOrb}, $P\in \Si_\la\cdot P_s$ for some $s<r$, hence $v_\bla\otimes w_P\in M_s\subseteq M_{r-1}$ by Lemma~\ref{LLevel}, giving a contradiction. 
So $r_k= s_k$ for all $k$. Then $E$ is of the form $\eta_\bnu$, and $E(v_\bla\otimes w_{P_r})\neq0$ implies $\bnu=\biota_i(\la+\eps_{P_r})$. 

The second statement now follows from Lemma~\ref{LYCrit}. 
\end{proof}

\begin{Theorem} \label{TOneColor} 
We have a filtration of $T(n,d+c)$-modules  
$$0=M_0\subseteq M_1\subseteq\dots\subseteq M_t=\De(\bla)\otimes\De(\bmu)
$$
such that $M_r/M_{r-1}\simeq \De(\biota_i(\la+\eps_{P_r}))$ for all $r=1,\dots,t$.
\end{Theorem}
\begin{proof}
We consider the filtration (\ref{EMFilt}). By Lemma~\ref{HiWtQuo}, each $M_r/M_{r-1}$ is a highest weight module of weight $\biota_i(\la+\eps_{P_r})$. 
Moreover, recalling that $\bla=\biota_i(\la)$ and $\bmu=\biota_i(\om_c)$, by  Theorem~\ref{TProdCharacter}, we have 
\begin{align*}
\ch\big(\De(\bla)\otimes\De(\bmu)\big)=\sum_{\nu\in\La_+(n,d+e)}c^{\,\nu}_{\la,\om_c}\ch\De(\biota_i(\nu))=\sum_{r=1}^t\ch\De(\biota_i(\la+\eps_{P_r})),
\end{align*}
where we have used Pieri's rule for the last equality. 
Therefore, using linear independence of characters, we get 
$$
\dim\big(\De(\bla)\otimes\De(\bmu)\big)=\sum_{r=1}^t\dim\De(\biota_i(\la+\eps_{P_r})).
$$
An application of Corollary~\ref{FiltCriteria} yields that each $M_r/M_{r-1}$ must be isomorphic to $\De(\biota_i(\la+\eps_{P_r}))$. 
\end{proof}

Recall (\ref{EDeMult}) and (\ref{EBoldLR}). 

\begin{Corollary} \label{CSmall} 
Let $n\in \Z_{>0}$, $d,c\in\Z_{\geq 0}$ with $d+c\leq n$, $\bla\in\La^I_+(n,d)$, $\bmu\in\La_+^I(n,c)$ and $\bnu \in \La^I_+(n,d+c)$. Then the tensor product $\De(\bla)\otimes\De(\bmu)$ has a standard filtration, and 
\[ (\De (\bla) \otimes \De (\bmu) : \De (\bnu)) = c_{\bla, \bmu}^{\bnu}.\]
\end{Corollary}
\begin{proof}
The first statement follows from Proposition~\ref{PRed} and Theorem~\ref{TOneColor}. The second statement now follows from Corollary ~\ref{DeChTens} using linear independence of formal characters.
\end{proof}

By a symmetric argument (switching the roles of $\X$ and $\Y$ everywhere), we also have the right module version of Corollary~\ref{CSmall}, which claims that the right $T(n,d+c)$-module $\De^\op(\bla)\otimes\De^\op(\bmu)$ has a $\De^\op$-filtration. 
In view of (\ref{ENabla}), by dualizing, we now get:

\begin{Corollary} \label{CNabla} 
Let $n\in \Z_{>0}$, $d,c\in\Z_{\geq 0}$ with $d+c\leq n$, $\bla\in\La^I_+(n,d)$, $\bmu\in\La_+^I(n,c)$ and $\bnu \in \La^I_+(n,d+c)$. Then the tensor product $\nabla(\bla)\otimes\nabla(\bmu)$ has a costandard filtration. 
\end{Corollary}

\begin{Remark} 
{\rm 
	Note that in Theorem~\ref{TOneColor}, the factors of the standard filtration are isomorphic to standard modules via {\em even} isomorphisms. Using this fact and (an appropriate strengthening of) Proposition~\ref{PRed}, one can similarly strengthen Corollaries~\ref{CSmall} and ~\ref{CNabla}.
}
\end{Remark}

\subsection{\boldmath The case of small $n$}
\label{SSTMainBig}
Let $d\in\Z_{\geq 0}$ and $n\in\Z_{>0}$. If $n<d$, the algebra $T(n,d)$ does not have to be quasihereditary, but it still has a natural family of `standard' and  `costandard' modules which play an important role. For example, if $A$ has a standard anti-involution then $T(n,d)$ is cellular with `standard' modules being the cell modules, see \cite[Lemma 6.25]{KMgreen3}. These `standard' (resp. `costandard') modules are obtained by an idempotent truncation from the standard modules $\De(\bla)$ (resp. costandard modules $\nabla(\bla)$) over $T(N,d)$ for any $N\geq d$. We now provide details on that. 

Throughout the subsection we assume that $N\geq n$. In view of  Lemma~\ref{LIdEasy}, we now always identify the algebras $T(n,d)$ and  
$\eta^{N}_n(d)T(N,d)\eta^{N}_n(d)$. For any $T(N,d)$-module $V$, we consider $\eta^{N}_n(d)V$ as a module over $T(n,d)=\eta^{N}_n(d)T(N,d)\eta^{N}_n(d)$.

We always consider $\La_+^I(n,d)$ as a subset of  $\La_+^I(N,d)$ by adding $N-n$ zeroes to every component $\la^{(i)}$ of $\bla=(\la^{(0)},\dots,\la^{(l)})\in\La_+^I(n,d)$. Note that  this embedding is a bijection if $n\geq d$. However, when we consider $\bla\in\La_+^I(n,d)$ as an element of $\La_+^I(N,d)$ the set $\Std^X(\bla)$ of standard $X$-colored $\bla$-tableaux changes, so in this subsection we will use the more detailed notation $\Std^X_n(\bla)$ to indicate that the entries of the tableaux are of the form $r^x$ with $r\in[n]$. We will also use the 
more detailed notation $\De_n(\bla)$ for the standard $T(n,d)$-module $\De(\bla)$ which so far has only been defined for all $\bla\in\La_+^I(n,d)$ when $n\geq d$. Recall from (\ref{EBasisDe}) that for $n\geq d$ we have that $\De_n(\bla)$ has basis $\{v_\T:=\X_\T v_\bla\mid \T\in\Std^X_n(\bla)\}$.

Let $N\geq d$. Fix $\bla\in \La_+^I(N,d)$. Recall the idempotent $\eta^N_n(d)$ of (\ref{E130818}). It is easy to see that for $\T\in\Std_N^X(\bla)$, we have 
\begin{equation}\label{E151020_2}
\eta^N_n(d)v_\T=
\left\{
\begin{array}{ll}
	v_\T &\hbox{if $\T\in\Std_n^X(\bla)$},\\
	0 &\hbox{otherwise.}
\end{array}
\right.
\end{equation}

If $n\geq d$ it follows from (\ref{E151020_2}) that $\dim \De_n(\bla)=\dim \eta^N_n(d)\De_N(\bla)$. Since the $T(n,d)$-module $\eta^N_n(d)\De_N(\bla)$ is easily seen to be a highest weight module of weight $\bla$, Proposition~\ref{PUn} now yields an isomorphism of $T(n,d)$-modules
\begin{equation}\label{E151020}
\De_n(\bla)\simeq \eta^N_n(d)\De_N(\bla).
\end{equation}

Now for $n<d$ and $\bla\in\La_+^I(N,d)$, we define the `standard' module
$$
\De_n(\bla):=\eta^N_n(d)\De_N(\bla). 
$$ 
By (\ref{E151020}), this definition does not depend on the choice of $N\geq d$. However, note that some of the $\De_n(\bla)$'s might be zero. Define 
$$
\Par_+^X (n,d) := \{ \bla \in \La_+^I(N,d) \mid \Std^X_n(\bla) \neq \varnothing\}. 
$$
Note that $\Par_+^X (n,d)$ does not depend on the choice of $N\geq d$. Moreover, 
$$\La_+^I(n,d)\subseteq \Par_+^X (n,d)\subseteq \La_+^I(N,d),$$ with containments being equalities when $n\geq d$. By (\ref{E151020_2}), we have:

\begin{Lemma} 
Let $N\geq d>n$ and $\bla\in \La_+^I(N,d)$. Then $\De_n(\bla)\neq 0$ if and only if $\bla\in \Par_+^X (n,d)$. 
\end{Lemma}

The story for the costandard modules $\nabla_n(\bla):=\eta^N_n(d)\nabla_N(\bla)$ is entirely similar, the non-zero ones being labeled by $\Par_+^Y (n,d) := \{ \bla \in \La_+^I(N,d) \mid \Std^Y_n(\bla) \neq \varnothing\}.$

\begin{Theorem} 
Let $\bla\in \Par_+^X (n,d)$ and $\bmu\in \Par_+^X (n,c)$. Then the $T(n,d+c)$-module $\De_n(\bla)\otimes \De_n(\bmu)$ has a filtration with factors of the form $\De_n(\bnu)$ with $\bnu\in \Par_+^X (n,d+c)$. Similarly for $\bla\in \Par_+^Y (n,d)$ and $\bmu\in \Par_+^Y (n,c)$, the $T(n,d+c)$-module $\nabla_n(\bla)\otimes \nabla_n(\bmu)$ has a filtration with factors of the form $\nabla_n(\bnu)$ with $\bnu\in \Par_+^Y (n,d+c)$.
\end{Theorem}
\begin{proof}
We prove the result for the $\De$'s, the proof for $\nabla$'s being similar. 
Choose $N\geq d+c$. By Corollary~\ref{CSmall}, $\De_N(\bla)\otimes \De_N(\bmu)$ has a filtration with factors of the form $\De_N(\bnu)$ with $\bnu\in \La_+^I (N,d+c)$. Applying the exact functor 
$$\mod{T(N,d+c)}\to \mod{T(n,d+c)},\ V\mapsto \eta^N_n(d)V$$ 
to this filtration and using Lemma~\ref{LTensTrunc}, we get the required result.
\end{proof}

\end{document}